\documentclass[12pt,a4paper]{article}
\usepackage{amsmath}
\usepackage{amssymb}

\usepackage{mathtools}
\usepackage{amsthm}
\usepackage[usenames]{color}

\usepackage{hyperref}
\usepackage{cleveref}

\usepackage{color}

\usepackage{graphicx}

\usepackage{stmaryrd}

\usepackage[numbers,sort&compress]{natbib}
\usepackage{doi}   
\usepackage{url}

\DeclareFontFamily{U}{BOONDOX-calo}{\skewchar\font=45 }
\DeclareFontShape{U}{BOONDOX-calo}{m}{n}{
  <-> s*[1.05] BOONDOX-r-calo}{}
\DeclareFontShape{U}{BOONDOX-calo}{b}{n}{
  <-> s*[1.05] BOONDOX-b-calo}{}
\DeclareMathAlphabet{\mathcalb}{U}{BOONDOX-calo}{m}{n}
\SetMathAlphabet{\mathcalb}{bold}{U}{BOONDOX-calo}{b}{n}
\DeclareMathAlphabet{\mathbcalb}{U}{BOONDOX-calo}{b}{n}

\usepackage{urwchancal}
\DeclareFontFamily{OT1}{pzc}{}
\DeclareFontShape{OT1}{pzc}{m}{it}{<-> s * [1.10] pzcmi7t}{}
\DeclareMathAlphabet{\mathcalc}{OT1}{pzc}{m}{it}

\allowdisplaybreaks

\newtheorem{theorem}{Theorem}
\newtheorem{lemma}[theorem]{Lemma}

\newtheorem{proposition}[theorem]{Proposition}
\newtheorem{corollary}[theorem]{Corollary}
\newtheorem{remark}[theorem]{Remark}

\newtheorem{fact}[theorem]{Fact}

\newcommand{\ind}{{\mathbf 1}}

\newcommand{\E}{{\mathbb E}}
\newcommand{\Z}{{\mathbb Z}}
\newcommand{\R}{{\mathbb R}}
\newcommand{\N}{{\mathbb N}}
\newcommand{\p}{{\mathbb P}}

\newcommand{\bL}{{\mathbb L}}
\newcommand{\calM}{\mathcal{M}}

\newcommand{\calF}{{\mathcal{F}}}

\newcommand{\bbP}{\mathbb{P}}
\newcommand{\bbE}{\mathbb{E}}
\newcommand{\bbL}{\mathbb{L}}
\newcommand{\bbN}{\mathbb{N}}
\newcommand{\bbR}{\mathbb{R}}

\renewcommand{\ell}{\mathcalb{l}}

\newcommand{\gG}{\Gamma}
\newcommand{\gd}{\delta}
\newcommand{\dd}{\text{\rm d}}

\newcommand{\lar}{\vartriangleleft}
\newcommand{\rar}{\vartriangleright}

\newcommand{\logZ}{\mathbf{z}}
\newcommand{\bfz}{\mathbf{z}}
\newcommand{\gz}{\zeta}

\newcommand{\ttt}{\mathtt{t}}

\newcommand{\tu}{\mathtt{u}}
\newcommand{\tuplus}{\mathtt{u}^+}

\newcommand{\lbra}{\llbracket}
\newcommand{\rbra}{\rrbracket}

\newcommand{\bbZ}{{\ensuremath{\mathbb Z}} }

\title{Further results on the free energy of the random field Ising chain in the case of centered disorder}
\author{Orph\'ee Collin$^{1}$}

\begin{document}

\maketitle

{\footnotesize 
\noindent $^{~1}$Technische Universit\"at Wien, Austria
\\
\noindent  e-mail:
\texttt{orphee.collin@normalesup.org}
}

\begin{abstract}
	We study the expansion of the limiting free energy density of the random field Ising chain with centered IID disorder and homogeneous coupling parameter, when the latter goes to infinity. We extend the first order result of \cite{C25} to the general case of finite second moment. Furthermore, we identify the value of the unknown constant appearing in the main theorem of \cite{CGGH25}, which is the first non-universal constant in the expansion of the limiting free energy density.
\\[.3cm]\textbf{AMS subject classifications (2010 MSC):}
34D08, 
60K37, 
82B44, 
60K35. 
\\[.3cm]\textbf{Keywords:} random field Ising chain, disordered systems,  random matrix products
\\[.3cm] This work is licensed under a Creative Commons Attribution | 4.0 International licence (CC-BY 4.0, https://creativecommons.org/licenses/by/4.0/).
\end{abstract}

\section{Introduction}\label{sec:introduction}


\subsection{Overview}



Let  $\mu$ be a law on $\bbR$ such that
\begin{equation}\label{eq:assump_mu_initial}
	\vartheta^2:= \int_\R x^2 \mu(\dd x) \in (0, \infty)\, , \qquad \int_\R x \mu(\dd x)=0 \, ,
\end{equation}
i.e., $\mu$ has a finite second moment, is centered and different from $\delta_0$, and its variance is denoted $\vartheta^2$.
We consider, under a probability measure denoted $\bbP$, an i.i.d.~sequence $h=(h_n)_{n\in\bbN}$ of real random variables following law $\mu$. We further consider a real positive parameter $J$.

The object of this article is to study, as $N$ goes to infinity, the product $M_1\dots M_N$ of the random matrices
\begin{equation}\label{eq:defM_n}
  M_n=
  \begin{pmatrix}
    e^{h_n} & e^{-2J+h_n}\\
    e^{-2J-h_n} & e^{-h_n}
  \end{pmatrix}
  \, , \qquad  n=1,2, \dots
\end{equation}
Following Furstenberg's theory (as will be detailed in Section \ref{sec:definitions}), $\bbP$-almost surely, the four coefficients of the matrix  $M_1\dots M_N$ all have the same exponential growth rate as $N$ goes to infinity, this rate is deterministic, it is called the Lyapunov exponent of this random matrix product and we will denote it by $\calF(J)$.


\smallskip 

Our first main result (Corollary \ref{th:cor_first_order}) states that under assumption \eqref{eq:assump_mu_initial}, as $J\to\infty$, 
\begin{equation}
\calF(J) \sim \frac{\vartheta^2}{2J}\, .
\end{equation}

In \cite{CGGH25} it was shown under some regularity assumption and under stronger integrability assumption on the law $\mu$ that the Lyapunov exponent $\calF(J)$ has the following asymptotic expansion as $J\to \infty$:
\begin{equation}\label{eq:thCGH25_simp}
\calF(J) 
= \frac{\vartheta^2}{2J+ \kappa} 
+ R(J)\, , 
\end{equation}
where the decay of the rest $R(J)$ depends on the integrability assumption on $\mu$: see Section \ref{sec:theorems} for a detailed statement.
The expression given in \cite{CGGH25} for the constant $\kappa$ is not explicit.
The second main contribution of the present work (Corollary \ref{th:cor_dev_free_ener}) is to provide a more explicit expression for $\kappa$. This will be done by establishing an expansion up to order $1/J^2$ of $\calF(J)$, as $J\to\infty$.
Our technique does not go beyond the order $1/J^2$, so it does not enlighten the remarkable simplicity of the expansion in \eqref{eq:thCGH25_simp}.

Our approach relies on the physical origins of the problem. The random matrix product indeed expresses the partition function of the random field Ising chain (RFIC) and the Lyapunov exponent $\calF(J)$ can be identified as the limiting free energy density of the RFIC.

We will proceed in two steps. First we will estimate the limiting maximal energy density $\mathcal{M}(J)$ of the RFIC (defined in Lemma \ref{th:def_free_ener} below). Then we will relate the limiting free energy density $\calF(J)$ with the limiting maximal energy density $\mathcal{M}(J)$: inspired by the description of the typical configurations given for example in \cite{CGH25+}, we will estimate the partition function of the RFIC by considering the contributions given by those configurations which have an energy close to the maximum.

\subsection{Definitions and basic facts}\label{sec:definitions}

For a given realisation of the sequence $h=(h_1, \dots, h_N)$, for every positive integer $N$, every positive real number $J$, and every pair $(g, d) \in\{+1, -1\}^2$, the Ising model on the chain $\{1, \dots, N\}$, with coupling constant $J$, external field $h$, and boundary conditions $(g,d)$ is usually defined by the following Hamiltonian:
\begin{equation}
\begin{aligned}
\{+1, -1\}^N & \to \R\\
\sigma=(\sigma_1, \dots, \sigma_N) & \mapsto J \sum_{n=0}^{N} \sigma_n\sigma_{n+1} + \sum_{n=1}^N h_n \sigma_n
\end{aligned}
\end{equation}
with the convention $\sigma_0=g, \sigma_{N+1}=d$.

We decide to subtract $(N+1)J$ to the Hamiltonian so, using the relation $\sigma_n \sigma_{n+1} = 1 - 2 \ind_{\sigma_n \neq \sigma_{n+1}}$, we rather consider the following Hamiltonian:
\begin{equation}
\begin{aligned}
H_{N, J, h}^{gd}:  \qquad \qquad \{+1, -1\}^N & \to \R\\
\sigma=(\sigma_1, \dots, \sigma_N) & \mapsto  -2J \sum_{n=0}^{N} \ind_{\sigma_n \neq  \sigma_{n+1}} + \sum_{n=1}^N h_n \sigma_n
\end{aligned}
\end{equation}
still with the convention $\sigma_0=g, \sigma_{N+1}=d$.

Each \emph{configuration} $\sigma\in \{+1, -1\}^N$ of \emph{spins} $-1$ or $+1$ on the \emph{system} $\{1, \dots, N\}$ is thus assigned a value $H_{N, J, h}^{gd}(\sigma)$, which we will also refer to as its \emph{energy}.
The partition function associated with this Hamiltonian is:
\begin{equation}\label{eq:def_parti_func}
Z_{N, J, h}^{gd}= \sum_{\sigma\in  \{+1, -1\}^N } \exp\left( H_{N, J, h}^{gd}(\sigma)\right)\, .
\end{equation}

A central object of study in statistical physics is the limiting free energy density associated with a Hamiltonian. Let us introduce it by the following lemma, as well as another quantity of interest for us: the limiting maximal energy density. 

For a given realisation of the sequence $h=(h_1, \dots, h_N)$, for every positive integer $N$, every positive real number $J$, and every pair $(g, d) \in\{+1, -1\}^2$, we consider the maximum of the Hamiltonian (or maximal energy):
\begin{equation}
\begin{aligned}
M_{N, J, h}^{gd} := \sup_{\sigma \in \{+1, -1\}^N } M_{N, J, h}^{gd} (\sigma).
\end{aligned}
\end{equation}

For the lemma, we only require that $\mu$ (sometimes refered to as the disorder law) admits a finite first moment.

\begin{lemma}\label{th:def_free_ener}
Assume that $\mu$ has a finite first moment. For fixed $J\in(0, \infty)$ and $(g,d)\in\{+1, -1\}^2$, for almost every realization of the disorder sequence $h$ (i.e., $\bbP$-almost surely), there is convergence as $N\to\infty$ of the free energy density
\begin{equation}
\frac{1}{N} \log \left( Z_{N, J, h}^{gd} \right)
\end{equation}
and of the maximal energy density
\begin{equation}
\frac{M_{N, J, h}^{gd}}{N} 
\end{equation}
towards deterministic quantities denoted respectively $\calF(J)$ and  $\calM(J)$, which depend on the law $\mu$, the coupling intensity $J$, but not on the boundary conditions $(g,d)$. The convergences also hold in $\bL^1$.
\smallskip 
\end{lemma}
The quantity $\calF(J)$ is called the limiting free energy density, we will sometimes call it simply the free energy. 
We call the quantity $\mathcal{M}(J)$ the limiting maximal energy density.

\begin{proof} We work with fixed $J\in(0, \infty)$ and start by treating the case of boundary conditions $++$. In this proof, we use the abbreviations: $Z_{N, h}=Z_{N, J, h}^{++}$ and $M_{N,h}=M_{N, J, h}^{++}$. We will use Kingman's subadditive ergodic theorem. Indeed, the application
\begin{equation}
\begin{aligned}
\Theta : \R^\N   & \to \R^\N \, , \\
h=(h_n)_{n\in\N} & \mapsto \Theta h = (h_{n+1})_{n\in\N}\, ,
\end{aligned}
\end{equation}
and the law $\mu^{\otimes \N}$ define an ergodic system and, for every positive integers $N$ and $M$, using the lower bound
\begin{equation}\label{eq:subadd_H}
  \begin{aligned}
    H_{N+M, J, h}^{++} & ((\sigma_{1}, \dots, \sigma_{N+M}))  \\
    & \ge H_{N, J, h}^{++}((\sigma_1, \dots, \sigma_N)) - 4J + H_{M, J, \Theta^N h}^{++}((\sigma_{N+1}, \dots, \sigma_{N+M}))\, ,
  \end{aligned}
\end{equation}
we have
\begin{equation}
Z_{N+M, h}\ge Z_{N, h} e^{-4J} Z_{M, \Theta^N h}  \, ,
\end{equation}
and 
\begin{equation}
M_{N+M, h}\ge  M_{N, h} -4J + M_{M,  \Theta^N h}  \, .
\end{equation}
Kingman's subadditive ergodic theorem then gives the almost sure and $\bL^1$ convergence of
\begin{equation}
\frac{1}{N} \log \left( Z_{N, h} e^{-4J}\right)
\end{equation}
and of 
\begin{equation}
\frac{M_{N, h}-4J}{N} 
\end{equation}
towards deterministic quantities. This yields the lemma in the case of boundary conditions ++.

To treat other boundary conditions $(g,d)\in \{+1, -1\}^2$, observe that for every configuration $\sigma\in \{+1, -1\}^N$
\begin{equation}\label{eq:encadr_Hgd}
H_{N, J, h}^{++} (\sigma) - 4J \le H_{N, J, h}^{gd} (\sigma) \le H_{N, J, h}^{++} (\sigma) + 4J \, , 
\end{equation}
so that
\begin{equation}
e^{- 4J} Z_{N, J, h}^{++} \le Z_{N, J, h}^{gd} \le e^{4J} Z_{N, J, h}^{++}\, ,
\end{equation}
and 
\begin{equation}
 M_{N, J, h}^{++} - 4J 
 \le M_{N, J, h}^{gd} 
 \le M_{N, J, h}^{++} + 4J \, .
\end{equation}
\end{proof}

\bigskip

Let us now relate the partition function and the free energy of the RFIC with the random matrix product in \eqref{eq:defM_n}. We introduce the matrices
\begin{equation}\label{eq:intro1_en:def_matrices}
Q_J :=
\begin{pmatrix}
1 & e^{-2J}\\
e^{-2J} & 1
\end{pmatrix} \, ,
\qquad
T_x :=
\begin{pmatrix}
e^{x} & 0\\
0 & e^{-x}
\end{pmatrix} \, ,  x\in\R
\end{equation}
and view them as indexed by $\{+1, -1\}$. Then the partition function $Z_{N, J, h}^{gd}$ is equal to the coefficient indexed by $(g, d)$ of the matrix product:
\begin{equation}\label{eq:matrixproduct_complete}
Q_J T_{h_1} Q_J T_{h_2} Q_J \dots Q_J T_{h_n} Q_J\, .
\end{equation}
Indeed, when computing the $(g,d)$-coefficient, one has to take into account the contributions to the matrix product of all the paths of length $2N+1$ with values in $\{+1, -1\}$ starting at $g$ and ending at $d$. When going through a matrix of the type $T_x$, the path should actually not switch (otherwise it has weight 0), and gains a factor $e^{x}$ or $e^{-x}$ depending on its value. When going trough a matrix $Q_J$, the path can switch to opposite value and then it gets a factor $e^{-2J}$. This yields exactly $Z_{N, J, h}^{gd}$.

We deduce that $\calF(J)$ is, almost surely, the exponential growth rate of the coefficients of the random matrix product in \eqref{eq:matrixproduct_complete}.
Removing the left-most matrix $Q_J$ there does not change that fact and this yields that $\calF(J)$ is also the almost sure exponential growth rate of the coefficients of $M_1\dots M_N$ with
\begin{equation}
M_n :=   T_{h_n} Q_J
=
\begin{pmatrix}
e^{h_n} & e^{-2J+h_n}\\
e^{-2J-h_n} & e^{-h_n}
\end{pmatrix} \, , \qquad n\ge 1\, .
\end{equation}
Hence, $\calF(J)$ is the so-called Lyapunov exponent of the product $M_1\dots M_N$.

\begin{remark}
  For completeness, let us present the definition of the Lyapunov exponent in a general setup. If $(M_n)_{n\in\N}$ is a random sequence of $n\times n$ matrices, we consider the following limit, if it exists:
  \begin{equation}
    \lim_{N\to\infty} \log |M_1\dots M_N | 
  \end{equation}
  where $|\cdot|$ is any norm on $\calM_{n}(\R)$; the limit being independent of the choice of the norm, due to the equivalence of norms in finite dimension. Choosing a submultiplicative norm, we have for every positive integers $N$ and $M$, 
  \begin{equation}
     \log |M_1\dots M_{N+M} | \le  \log |M_1\dots M_N | +  \log |M_{N+1}\dots M_{N+M} |  \, .
  \end{equation}
  Hence, if the sequence $(M_n)_{n\in\N}$ is i.i.d. and if $(\log|M_1|)_+$ has a finite expectation, then Kingman's subadditive theorem yields that the above limit almost surely exists, is deterministic, and that the convergence holds also in $\bL^1$.
  
  If furthermore all the coefficients of $M_1$ are almost surely positive, then Furstenberg's theory yields that all the coefficients of $M_1\dots M_N$ have the same exponential growth rate, equal to the Lyapunov exponent. 
\end{remark}



\subsection{Theorems}\label{sec:theorems}

\subsubsection{Results concerning the limiting maximal energy density}

\begin{theorem}\label{th:first_order_calM}
	Assume \eqref{eq:assump_mu_initial}. Then, as $J\to\infty$,
	\begin{equation}
		\calM(J) \sim \frac{\vartheta^2}{2J} \, .
	\end{equation}
\end{theorem}

\begin{theorem}\label{th:dev_calM}
	Assume \eqref{eq:assump_mu_initial} and that 
\begin{description}
\item[{(H-1)}]
there exists a positive integer $n_0$ such that $h_1+\dots+h_{n_0}$ has a density,
\end{description}
and either that
\begin{description}
	\item[{(H-2)}] there exists $\xi > 5$ such that $\int_\R |x|^\xi \mu(\dd x) < \infty$,
\end{description}
or that
\begin{description}
	\item[{(H-2$^\prime$)}] there exists $c>0$ such that  $\int_\R \exp(c|x|) \mu(\dd x)< \infty$.
\end{description}
Then there exists a positive constant $\widehat{\kappa}$ such that, as $J\to \infty$,
\begin{equation}
\label{eq:main}
\calM(J)
= \frac{\vartheta^2}{2J+ \widehat{\kappa}} 
+ 
\begin{cases}
O\left( 1/J^{\xi-4} \right) & \text{ under  {\bf (H-2)}},
\\
O\left(\exp\left(- \gd  J \right) \right) & \text{ under {\bf (H-2$^\prime$)}},
\end{cases} \,
\end{equation}
for a suitable choice of $\gd>0$.
\end{theorem}

We will provide a somewhat explicit expression for the value of the constant $\widehat{\kappa}$. 
To give this expression, let us introduce $S=(S_n)_{n\ge 0}$ the random walk associated to $h=(h_n)_{n\in \N}$:
\begin{equation}
S_0= 0 \, , \qquad S_n := \sum_{j=1}^n h_j \, , \qquad n= 1, 2, \dots 
\end{equation}
Further denote
\begin{equation}
\begin{aligned}
\alpha_\lar & :=  \inf\{ n \ge 1 : S_n > 0\} \, ,  \qquad H_\lar := S_{\alpha_\lar} \, ,\\
\alpha_\rar & :=  \inf\{ n \ge 1 : S_n < 0\} \, ,  \qquad H_\rar := -S_{\alpha_\rar} \, .\\
\end{aligned}
\end{equation}
The constant $\widehat{\kappa}$ is given by
\begin{equation}\label{eq:expr_widehat_kappa}
\widehat{\kappa}=\frac{1}{2} \left(\frac{\E[H_\lar^2]}{\E[H_\lar]}+\frac{\E[H_\rar^2]}{\E[H_\rar]}\right)\, .
\end{equation}

\subsubsection{Result concerning the free energy}\label{sec:intro_results_free_ener}


\begin{theorem}\label{th:dev_calF}
	Assume \eqref{eq:assump_mu_initial}. Then there exists a positive constant $\widetilde{\kappa}$ such that, as $J\to\infty$, 
\begin{equation}
\label{eq:th_relate_F_M}
\calF(J)
= \calM(J)+ \frac{\vartheta^2 \widetilde{\kappa}}{(2J)^2}
+ o(1/J^2).
\end{equation}
\end{theorem}

We will provide a somewhat explicit expression for the value of the constant $\widetilde{\kappa}$. 
To give this expression, we first need to introduce a process $Z=(Z_n)_{n\in \Z}$  indexed by $\Z$ and with values in $\R$. 
Let us first extend $(h_n)_{n\in \N}$ to $(h_n)_{n\in \Z}$, still an i.i.d. sequence with marginal law $\mu$, and extend the random walk $S=(S_n)_{n\ge 0}$ to the whole of $\Z$ by setting:
\begin{equation}
S_0=0\, , \qquad 
S_n = \sum_{i=1}^n h_i \text{ for } n\ge 1\, , \quad  
S_n= - \sum_{i=n+1}^0 h_i \text{ for } n\le -1 \, .
\end{equation}
In other terms, $S=(S_n)_{n\in\Z}$ is the walk satisfying $S_0=0$ and $S_n-S_{n-1}=h_n$ for all $n\in\Z$.
The process $Z=(Z_n)_{n\in \Z}$ is then obtained by conditioning $S$ to be non-negative: for a rigorous definition, see Section \ref{sec:around_extremum}, where we will also argue that $Z$ (respectfully, $(-Z_{-n})_{n\in\Z}$) locally describes the walk $S$ around a $\gG$-minimum (respectfully, a $\gG$-maximum). 
The constant $\widetilde{\kappa}$ is given by the following expectation:
\begin{equation}\label{eq:expr_widetilde_kappa}
\widetilde{\kappa} = \E \left[ \log\left( \sum_{n\in\Z} e^{-2Z_n} \right) \right]\, .
\end{equation}
Note that $\widetilde{\kappa}$ is always positive, since $Z_0=0$.

\subsection{Consequences and links with the existing literature}

From Theorems \ref{th:first_order_calM} and \ref{th:dev_calF} we derive the following corollary.
\begin{corollary}\label{th:cor_first_order}
	Assume \eqref{eq:assump_mu_initial}. Then, as $J\to\infty$,
	\begin{equation}
		\calF(J) \sim \frac{\vartheta^2}{2J} \, .
	\end{equation}
\end{corollary}
This first order expansion of $\calF(J)$ was already proven in \cite{C25} under the slightly stronger hypothesis that $\mu$ admits a moment of order $\xi$ for some $\xi>\frac{3+\sqrt{5}}{2}\approx 2.618$. The present paper thus extends this result to the most general case where it can be expected to hold: if $\mu$ does not have a finite second moment, then the variance $\vartheta^2$ is simply not defined. 
On the other hand, in \cite{C25}, a control was given on the rest, depending on the value of $\xi$, but let us stress that this control was not very precise: it did not go down to the order $J^{-2}$.

Furstenberg's theory for random matrix products provides an expression 
for the Lyapunov exponent $\calF(J)$ in terms of the invariant probability measure of some Markov chain on $\R$. In \cite{CGGH25}, exploiting this expression, it was established that under the same assumptions as in our Theorem \ref{th:dev_calM}, there exists $\kappa \in \bbR$ such that for $J\to \infty$
\begin{equation}
\label{eq:main_CGGH25}
\calF(J)
= \frac{\vartheta^2}{2J+ \kappa } 
+ 
\begin{cases}
O\left( 1/J^{\xi-4} \right) & \text{ under  {\bf (H-2)}},
\\
O\left(\exp\left(- \gd J\right) \right) & \text{ under {\bf (H-2$^\prime$)}},
\end{cases} \,
\end{equation}
for a suitable choice of $\gd>0$.

\medskip

The technique used in \cite{CGGH25} to derive this result yields expressions for the constants appearing in the numerator and in the denominator of \eqref{eq:main_CGGH25} which are not explicit (see formulas (3.13) and (3.17) in \cite{CGGH25}). The fact that the constant in the numerator is equal to the disorder variance $\vartheta^2$ derives from our Corollary \ref{th:cor_first_order}, but also already from \cite{C25}; and in fact in the appendix of \cite{CGGH25} the authors managed to show that their complicated expression for this constant can be reduced to $\vartheta^2$.
On the contrary the value of the coefficient $\kappa$ in the denominator remained implicit: its expression in \cite{CGGH25} involves two quantities depending on the invariant measures of some Markov chains on $\mathbb{R}$. Grasping the value of $\kappa$ is also beyond the reach of \cite{C25}, since, as already mentioned, the expansion proven there does not go down to the order $J^{-2}$. 

From Theorems \ref{th:dev_calM} and \ref{th:dev_calF} in the present work we derive a more explicit expression for the constant $\kappa$ in \eqref{eq:main_CGGH25}.
\begin{corollary}\label{th:cor_dev_free_ener}
	Under the assumptions of Theorem \ref{th:dev_calM}, the constant $\kappa$ appearing in \eqref{eq:main_CGGH25} is given by
	\begin{equation}
		\kappa= \widehat{\kappa} - \widetilde{\kappa}
	\end{equation}
	where $\widehat{\kappa}$ is given in \eqref{eq:expr_widehat_kappa} and $\widetilde{\kappa}$ in \eqref{eq:expr_widetilde_kappa}.
\end{corollary}
Let us stress that in the main theorem of \cite{CGGH25}, i.e., in \eqref{eq:main_CGGH25} above, $\kappa$ is relevant only if $\mu$ has a moment of order larger than 6.

Although we are unable to make explicit computations for any particular choice of the law $\mu$ (see Remark \ref{rem:no_computation}), this expression $\kappa=\widehat{\kappa}-\widetilde{\kappa}$ seems to depend on fine details of the law $\mu$, so it does not seem to exhibit the same universality as the constant in the numerator of \eqref{eq:main_CGGH25}.

\begin{remark}\label{rem:no_computation}
Unfortunately, we are unable to compute $\kappa$ in any special case using our formulas. 
If $\mu$ is the Laplace law, then we remark that $\widehat{\kappa}$ can be computed since, using the loss of memory property of exponential variables, $H_\lar$ and $H_\rar$ are distributed as exponential variables. The constant $\widetilde{\kappa}$ is however more complicated to compute. 

If $S$ is the simple random walk, then $\widehat{\kappa}= 1$, but we are unable to compute $\widetilde{\kappa}$; anyway, the simple random walk does not comply with hypothesis {\bf (H-1)} of Theorem \ref{th:dev_calM}.  

\end{remark}

Continuing our review of the literature, let us mention a special choice of the law $\mu$. In \cite{CGGH25} it is shown that if $\mu$ has density $\frac{1}{(2\cosh(x/2))^2}$, then $\kappa$ is equal to zero.



We mention that in \cite{DeMoor2025} Lyapunov exponents of more general products of i.i.d. random $2\times2$-matrices have been studied. A first order development of the Lyapunov exponent was obtained, which coincides with ours. In particular, the paper allows for negative coefficients in the matrices, which is well beyond our reach since we rely on lower and upper bounds that exploit the positivity of the terms contributing to the partition function. We stress that these results however do not go far beyond order $1/J$ and are established only for compactly supported disorder variables.

Finally, let us mention that beyond order $1/J^2$, the regularity assumptions on the disorder law are expected to play a role: in \cite{Luck1991}, it is indeed claimed that when $S$ is the (lazy) simple random walk, the next term in the expansion is of the form $\frac{\omega(J)}{J^3}$ where $\omega$ is a 1-periodic function. 


\bigskip

\subsection{Structure of the paper}

In Section \ref{sec:max_ener} we describe a configuration which maximises the Hamiltonian, with the help of the notion of $\gG$-extrema of the walk $S$ associated to $h$. Examining the behaviour of these $\gG$-extrema for large $\gG$, we obtain a proof of Theorem \ref{th:first_order_calM}.

In Section \ref{sec:proof_main_thms}, using the point of view of the large energy configurations of the RFIC, we prove Theorem \ref{th:dev_calF}. 

Finally, in Section \ref{sec:proof_th:dev_calM}, we prove Theorem \ref{th:dev_calM}. This last proof is essentially independent of the rest of the paper, in particular it is not based on the description of the configurations with large energy; it is modelled on the proof of the main theorem in \cite{CGGH25}.



\section{Maximal energy configuration}\label{sec:max_ener}

\subsection{Process of $\gG$-extrema associated to the random walk $S$}

We use the notations of \cite{CGH25+}, which we recall now. Recall that $S=(S_n)_{n\ge 0}$ is the random walk associated to the i.i.d. sequence $h=(h_n)_{n\in \N}$:
\begin{equation}
S_0= 0 \, , \qquad S_n := \sum_{j=1}^n h_j \, , \qquad n= 1, 2, \dots 
\end{equation}
Let us introduce for $0\le m\le n$
\begin{equation}
 S^\uparrow_{m,n}\,:=\,  \max_{m\le i\le j\le n} (S_j-S_i), \qquad S^\downarrow_{m,n}:= \max_{m\le i\le j\le n} (S_i-S_j)\,,  
 \end{equation}
 and for every $\gG>0$ the \emph{first time of $\gG$-decrease} 
\begin{equation}
 \label{eq:completenotation1}
 \ttt_1(\gG)\,:=\, \inf\left\{n > 0: S^\downarrow_{0,n} \ge \gG\right\} \, . 
 \end{equation}
For almost all realizations of $(h_n)$ we have that $\limsup_n S_n=-\liminf_n S_n= \infty$: we assume that we work on such realizations,
so the infimum in \eqref{eq:completenotation1} can be replaced by a minimum and
$\ttt_1(\gG)<\infty$.
Next we introduce (with the notation $\lbra j, k\rbra:=[j, k] \cap \bbZ$ for the integers $j\le k$)
\begin{equation}
 \label{eq:completenotation2}
\tu_1(\gG)\,:=\, \min\left\{ n \in \lbra 0,  \ttt_1(\gG)\rbra: S_n= \max_{i \in \lbra 0, \ttt_1(\gG)\rbra} S_i\right\}\,,
\end{equation}
so $\tu_1(\gG)$ is the first time that $S$ reaches its maximum before having a decrease of at least $\gG$: 
this is what we call \emph{location of the first  $\gG$-maximum}.
Note however that $S$ may reach this maximum also at later times and before time $\ttt_1(\gG)$. This is a situation we describe as the occurrence of \emph{multiple $\Gamma$-extrema}. In fact,
this happens with positive probability (at least for $\gG$ large) if and only 
if there exists a positive integer $n$ and real numbers $x_1,  \dots, x_n$ (not necessarily disctinct) which are atoms of the law of $h_1$ and such that $x_1+\dots+ x_n=0$. 
Therefore, we introduce also 
\begin{equation}
 \label{eq:completenotation2.1}
\tuplus_1(\gG)\,:=\, \max\left\{ n \in \lbra 0,  \ttt_1(\gG)\rbra:\,  S_n= \max_{i\in \lbra 0, \ttt_1(\gG)\rbra} S_i\right\}\,. 
\end{equation}
Now we proceed by looking for the first time of $\gG$-increase after $ \ttt_1(\gG)$
\begin{equation}
 \ttt_2(\gG)\,:=\, \min\left\{n> \ttt_1(\gG):\,  S^\uparrow_{\ttt_1(\gG),n} \ge \gG\right\} \, ,
 \end{equation}
 and we set 
 \begin{equation}
 \tu_2(\gG) \,:=\,  \min\left\{n\in \lbra \ttt_1(\gG), \ttt_2(\gG)\rbra:\,  S_n= \min_{i \in \lbra \ttt_1(\gG), \ttt_2(\gG)\rbra} S_i\right\}\,,
 \end{equation}
 and
 \begin{equation}
  \tuplus_2(\gG) \,:=\,  \max\left\{n \in \lbra \ttt_1(\gG), \ttt_2(\gG)\rbra:\,  S_n= \min_{i \in \lbra \ttt_1(\gG), \ttt_2(\gG)\rbra} S_i\right\}\, .
  \end{equation}
So  $\ttt_2(\gG)$ is the first time at which there is an increase of $S$  at least $\gG$ after the decrease time $\ttt_1(\gG)$.
And $ \tu_2(\gG)$, respectively $\tuplus_2(\gG)$, is the first (respectively last) absolute minimum of the walk between 
$\ttt_1(\gG)$ and $\ttt_2(\gG)$.

Now  (almost surely  in the realization of the $h$ sequence) we can  iterate this procedure to build the increasing sequences $(\ttt_j(\gG))_{j \in \bbN}$, 
$(\tu_j(\gG))_{j \in \bbN}$ and $(\tuplus_j(\gG))_{j \in \bbN}$ with 
\begin{equation}
\ttt_j(\gG) \, \le \tu_j(\gG)\, \le\,  \tuplus_j(\gG)\, < \, \ttt_{j+1}(\gG)\, ,
\end{equation} 
with  $\tu_j(\gG)$ and $ \tuplus_j(\gG)$ that are the locations of the (first and last) maxima (respectively minima) of $S$, more precisely the first and last $S$ absolute maxima in $[\ttt_j(\gG); \ttt_{j+1}(\gG)]$,
if $j$ is odd (respectively even). Explicitly 
\begin{equation}\label{eq:def_t_k}
 \ttt_{j+1}(\gG)\,:=\begin{cases} 
 \min\big\{n> \ttt_j(\gG):\,  S^\downarrow_{\ttt_j(\gG),n} \ge \gG\big\} & \textrm{ if $j$ is even} \, ,
 \\
 \min\big\{n> \ttt_j(\gG):\,  S^\uparrow_{\ttt_j(\gG),n} \ge \gG\big\} & \textrm{ if $j$ is odd} \, ,
 \end{cases}
 \end{equation}
  \begin{equation}
  \label{eq:new-u}
 \tu_{j+1}(\gG) \,:=\begin{cases}  
 \min\left\{n \in \lbra \ttt_j(\gG),  \ttt_{j+1}(\gG)\rbra:\,  S_n= \max_{i \in \lbra \ttt_j(\gG), \ttt_{j+1}(\gG)\rbra} S_i\right\} &  \textrm{ if $j$ is even}\, ,
 \\
  \min\left\{n \in \lbra \ttt_j(\gG),  \ttt_{j+1}(\gG)\rbra:\,  S_n= \min_{i \in \lbra \ttt_j(\gG), \ttt_{j+1}(\gG)\rbra} S_i\right\} &  \textrm{ if $j$ is odd}\, ,
 \end{cases}
 \end{equation}
 and $\tu_{j+1}^+(\gG)$ is defined like $\tu_{j+1}(\gG)$ with the minimum for 
 $n \in \lbra \ttt_j(\gG),  \ttt_{j+1}(\gG)\rbra$ replaced by a maximum. 
 
Like before, we  say that $\tu_{j}(\gG)$ and  $\tu_{j}^+(\gG)$ are location of $\gG$-extrema: maxima (respectively minima) if $j$ is odd (respectively even).
In $\lbra \tu_{j}(\gG), \tu_{j}^+(\gG)\rbra$ there can be other $\gG$-maxima if $j$ is odd (or $\gG$-minima if $j$ is even).

\bigskip

The notion of $\gG$-extrema is crucial to describe the configurations with maximal energy: in fact, this is done in detail in \cite{CGH25+} (see Appendix C there). Since taking into account the effects of the boundaries complicates the description, we are going to work only with boundary conditions $++$ and with special lengths of the system, to keep the presentation simple. This will be enough for our purpose. We say that a configuration $\sigma$ switches spin at position $n$ if $\sigma_n\neq \sigma_{n+1}$.

\begin{fact}\label{th:fact_maximal_config}
For every $K\ge0$, we set $N_K=t_{2K}(\gG)$ and we consider the system $\{1, \dots, N_K\}$. Then, the configuration which, with respect to boundary conditions $++$, switches spin exactly at every $u_k(\gG), k=1, \dots, 2K$, 
maximises the Hamiltonian $H_{N_K, J, h}^{++}$.
\end{fact}

\begin{proof}
A proof can be found in Appendix C of \cite{CGH25+}; however, we provide an alternate version here, which will help the reader follow the initial step of the proof of the upper bound in Theorem \ref{th:dev_calF} (Section \ref{sec:proof_upp_bound}). For convenience, in this proof we drop the dependence in $\gG$ in the notations, so we write $t_k$, $u_k$ and $u_k^+$ instead of $t_k(\gG)$, $u_k(\gG)$ and $u_k^+(\gG)$. We fix $K$ and denote $u_0=0$ and $u_{2K+1}= t_{2K}=N_K$. 
Below, we consider that $u_k$, the left-most $\gG$-extremum between $t_{k-1}$ and $t_k$, is \emph{the canonical} $k$-th $\gG$-extremum.
Let us show that starting from any configuration $\sigma$ on the system $\{1, \dots, N_K\}$ we can apply to it transformations that do not decrease its energy, and end up with the configuration described in Fact \ref{th:fact_maximal_config}.

First, we modify the configuration to make sure that for every $k=1, \dots, 2K$ there exists at most one wall located at a position between $u_{k-1}$ and $u_{k+1}^+$ (inclusive) having the same type as $u_k$ (i.e., a wall from +1 to -1 for $k$ odd, a wall from $-1$ to $+1$ for $k$ even).  If such a wall exists, we refer to it as the wall attached to $u_k$. 
Indeed, consider for example $k$ odd so that $u_k$ is a $\gG$-maximum; if there are at least two walls from + to - between $u_{k-1}$ and $u_{k+1}^+$ then there is a domain of - spins fully inside $\lbra u_{k-1} , u_k^+ \rbra$ or a domain of + spins fully inside  $\lbra u_k , u_{k+1}^+ \rbra$. 
Now, by construction, inside an ascending (respectivelly, descending) $\gG$-stretch, there is no drop (respectivelly, rise) of size $\gG$ or larger, thus removing this spin domain increases the energy.

Secondly, we make sure  that for every $k=1, \dots, 2K-1$, provided they exist, the wall attached to $u_k$ is on the left of the wall attached to $u_{k+1}$.
Indeed, consider for example $k$ odd so that $u_k$ is a $\gG$-maximum; if there are walls attached to $u_k$ (so a wall from + to -) and to $u_{k+1}$ (so a wall from - to +) and this walls are not ordered accordingly, then they delimit a + domain which is fully inside the descending $\gG$-stretch $\lbra u_k , u_{k+1}^+ \rbra$ and removing this domain increases the energy. 

Once these walls are correctly ordered, we note that moving them to the (canonical) $\gG$-extrema they are attached to does not decrease the energy — in fact, if they were not already located at one of the multiple $\gG$-extrema this increases the energy.

Finally, if there are $\gG$-extrema that have no wall attached to, then also one of the two neighboring extrema (or, for $k\in\{1, 2K\}$, the neighboring $\gG$-extremum) also has no wall attached: adding the two missing walls exactly at the locations of the $\gG$-extrema under consideration does not decrease the energy — in fact the energy increases if and only if the height of the corresponding $\gG$-stretch is strictly larger than $\gG$. We end up with the configuration described in Fact \ref{th:fact_maximal_config}.
\end{proof}

Next, we explain how the sequence of $\gG$-extrema breaks the walk $S$ into independent portions. The times $t_k(\gG)$ we have introduced in \eqref{eq:def_t_k} are of course stopping times. In fact, using renewal theory, we obtain the very useful fact that the times $u_k(\gG)$ and $u_k^+(\gG)$, although they are not stopping times, are still renewal times.

\begin{fact}\label{th:fact_indep_portions}
	\begin{itemize}
	\item The portions 
$(S_{t_{k-1}(\gG)+n}-S_{t_{k-1}(\gG)})_{0\le n\le  u_k(\gG) - t_{k-1}(\gG) }$ 
for $k=1, 3, 5, \dots$ (with $t_0(\gG)=0$) are identically distributed.
\item The portions 
$(S_{u_k(\gG)+n}-S_{u_k(\gG)})_{0 \le n \le  u_{k}^+(\gG) - u_k(\gG) }$ 
for $k=1,3, 5, \dots$ are identically distributed.
\item The portions 
$(S_{u_k^+(\gG)+n}-S_{u_k^+(\gG)})_{0 \le n \le  t_{k}(\gG) - u_k^+(\gG) }$ 
for $k=1,3, 5, \dots$ are identically distributed.
\item The portions 
$(S_{t_{k-1}(\gG)+n}-S_{t_{k-1}(\gG)})_{0\le n\le  u_k(\gG) - t_{k-1}(\gG) }$ 
for $k=2, 4, 6, \dots$ (with $t_0(\gG)=0$) are identically distributed.
\item The portions 
$(S_{u_k(\gG)+n}-S_{u_k(\gG)})_{0 \le n \le  u_{k}^+(\gG) - u_k(\gG) }$ 
for $k=2, 4, 6, \dots$ are identically distributed.
\item The portions 
$(S_{u_k^+(\gG)+n}-S_{u_k^+(\gG)})_{0 \le n \le  t_{k}(\gG) - u_k^+(\gG) }$ 
for $k=2, 4, 6, \dots$ are identically distributed.
	\end{itemize}

Furthermore, all the above portions are independent. 

\end{fact}


An important consequence for us is that:
\begin{itemize}
	\item the portions 
$(S_{u_k(\gG)+n}-S_{u_k(\gG)})_{0 \le n \le  u_{k+1}(\gG) - u_k(\gG) }$ 
for $k=1, 3, 5, \dots$, which we call the descending $\gG$-stretches, are identically distributed;
\item the portions 
$(S_{u_k(\gG)+n}-S_{u_k(\gG)})_{0 \le n \le  u_{k+1}(\gG) - u_k(\gG) }$ 
for $k=2, 4, 6, \dots$, which we call the ascending $\gG$-stretches, are identically distributed;
\end{itemize}
and all those portions, generically called the $\gG$-stretches, are independent.


\begin{proof}
	We use fluctuation theory for the random walk $S$. We define the following stopping times:
	\begin{equation}
		\rho_0:=0 \, , 
		\quad \text{ and, for } k\ge 0 ,  \quad 
		\rho_{k+1}:=
		\inf\{n>\rho_k : S_n\ge S_{\rho_k}\}\, ,
	\end{equation}
which are called the weak ascending record times. They decompose the walk $S$ in i.i.d. portions, called excursions:
	\begin{equation}
		\left(S_{\rho_k+n}-S_{\rho_k}\right)_{0\le n \le \rho_{k+1}-\rho_k}\, , 
		\qquad k=1, 2, \dots
	\end{equation}
	Then the first occurence of a drop of size at least $\gG$ in $S$ is determined by the first excursion with depth $\gG$ or larger. Let us denote by $p_\gG$ the probability for an excursion to have a depth not smaller than $\gG$. The random walk up to time $t_1(\gG)$ can be sampled as follows. Toss a coin:
	\begin{itemize}
		\item with probability $p_\gG$, decide that the first excursion has depth not smaller than $\gG$ and sample it conditionally on this fact; then repeat the procedure to sample the next excursions. 
		\item with probability $1-p_\gG$, decide that the first excursion has depth not smaller than $\gG$ and sample it only partially: sample it up to the point where it reaches depth $\gG$, i.e., sample a realisation of the walk conditioned on hitting $(-\infty, -\gG]$ before returning in $[0, \infty)$. 
	\end{itemize} 
	The end of this partial excursion is the time $t_1(\gG)$, and its starting point is $u_1^+(\gG)$. 
	In view of this sampling procedure, we get that the portions of walk from $0$ to $u_1^+(\gG)$ and from $u_1^+(\gG)$ to $t_1(\gG)$ are independent. 
	
	Considering the excursions that preceed time $u_1^+(\gG)$, we can in the same spirit break the portion from $0$ to $u_1^+(\gG)$ into two independent portions: from $0$ to $u_1(\gG)$ and from $u_1(\gG)$ to $u_1^+(\gG)$. 
	
	After the stopping time $t_1(\gG)$, we repeat the same arguments, alternating weak descending record times and weak ascending record times.
\end{proof}

\begin{remark}\label{rem:LLN}
	Throughout Sections~\ref{sec:max_ener} and~\ref{sec:proof_main_thms}, we will repeatedly use the law of large numbers (LLN), taking advantage of the i.i.d.\ structure provided by Fact~\ref{th:fact_indep_portions}. We record here two elementary but useful observations.
	
	\medskip
	\noindent
	\textbf{First,} if \((X_n)\) is a sequence of non-negative random variables all stochastically dominated by a common integrable random variable, then almost surely
	\[
	\frac{X_n}{n} \longrightarrow 0 \quad \text{as } n \to \infty.
	\]
	An illustrative application of this is the following: since \(t_{2K}(\gG)\) is the  sum of the i.i.d.\ random variables \(t_{2k}(\gG) - t_{2(k-1)}(\gG)\) for \(k = 1, \dots, K\) (denoting $t_0(\gG)=0$), the LLN yields the almost sure convergence
	\[
	\frac{t_{2K}(\gG)}{K} \to \mathbb{E}[t_2(\gG)] \quad \text{as } K \to \infty,
	\]
	but the same holds also for the ratios \(\frac{u_{2K}(\gG)}{K}\), \(\frac{t_{2K+1}(\gG)}{K}\) and \(\frac{u_{2K+1}(\gG)}{K}\).

	\medskip
	\noindent
	\textbf{Second,} the independence assumption in the LLN can sometimes be relaxed. Suppose \((X_n)_{n \geq 1}\) is a sequence of identically distributed integrable random variables such that the subsequences \((X_{2n})_{n \geq 0}\) and \((X_{2n+1})_{n \geq 0}\) are each independent. Then, applying the LLN separately to both subsequences yields the LLN for the full sequence.
\end{remark}

We introduce the following (positive) variables:
\begin{equation}
	H_\gG^\downarrow = S_{u_1(\gG)} - S_{u_2(\gG)} \, ,  \qquad  
	H_\gG^\uparrow = S_{u_3(\gG)} - S_{u_2(\gG)} \, , 
\end{equation}
\begin{equation}
	L_\gG^\downarrow = u_2(\gG) - u_1(\gG) \, ,  \qquad  
	L_\gG^\uparrow = u_3(\gG) - u_2(\gG) \, ,
\end{equation}
which are distributed as the heights and lengths of generic descending and ascending $\gG$-stretches. We note that, due to Fact \ref{th:fact_indep_portions}, $u_1(\gG)$ is distributed as $u_3(\gG)-t_2(\gG)$, so $L_\gG^\downarrow + L_\gG^\uparrow$ is distributed as $t_2(\gG)$.



Using the LLN, we derive the following formula for the limiting maximal energy density. 

\begin{corollary}\label{th:cor_expression_calM}
	For every $J>0$, we have, recalling that $\gG=2J$, 
	\begin{equation}
		\calM(J)=
		\frac{\E\left[H_\gG^\downarrow + H_\gG^\uparrow \right] - 2\gG }
		{\E\left[L_\gG^\downarrow + L_\gG^\uparrow\right]} \, .
	\end{equation}
\end{corollary}

\begin{remark}
	In the case where the law $\mu$ is symmetric, the couple $(H_\gG^\downarrow, L_\gG^\downarrow )$ is distributed as  $(H_\gG^\uparrow, L_\gG^\uparrow )$ and the above formula takes the slightly simpler form:
	\begin{equation}
		\calM(J)=
		\frac{\E\left[H_\gG^\downarrow\right] - \gG}
		{\E\left[L_\gG^\downarrow\right]} \, .
	\end{equation}

	If $\mu$ is not symmetric, we cannot a priori claim that the couple $(H_\gG^\downarrow, L_\gG^\downarrow )$ is distributed as $(H_\gG^\uparrow, L_\gG^\uparrow )$. 
	However, it still holds that 
	$\E\left[H_\gG^\uparrow\right] = \E\left[H_\gG^\downarrow\right]$. Indeed, by the LLN $S_n/n$ converges almost surely towards 0 as $n\to\infty$, in particular so does $S_{u_{2K}(\gG)}/u_{2K}(\gG)$; but using the LLN twice, this ratio also converges towards the ratio $\frac{E\left[H_\gG^\downarrow - H_\gG^\uparrow\right]}{\E\left[ L_\gG^\downarrow + L_\gG^\uparrow \right]}$, whose numerator must therefore be zero.
\end{remark}

\begin{proof}[Proof of Corollary \ref{th:cor_expression_calM}]
	Let us take $K\in \N$ and $N=N_K=t_{2K}(\gG)$. We use the configuration given in Fact \ref{th:fact_maximal_config}, which maximises the Hamiltonian. For this configuration, we have
	\begin{equation}
		M_{N_K, J, h}^{++}=H_{N_K, J, h}^{++}(\sigma) 
		= -\Gamma \times  2K + S_{u_1}
		+ \left( \sum_{1\le k \le 2K-1} |S_{u_{k+1}}-S_{u_k}| \right)
 		+ S_{t_{2K}}-S_{u_{2K}} \, .
	\end{equation}
The variables $ |S_{u_{k+1}}-S_{u_k}|, k=1, 3, \dots$ being i.i.d., as well as the variables $ |S_{u_{k+1}}-S_{u_k}|, k=2, 4, \dots$, the LLN yields that, almost surely,
	\begin{equation}
		\frac{M_{N_K, J, h}^{++}}{K} 
		\underset{K\to\infty}\longrightarrow 
		-2\Gamma  
		+ \E[H_\gG^\downarrow+H_\gG^\uparrow] \, .
	\end{equation}
Furthermore, as already stressed in Remark \ref{rem:LLN}, almost surely,
	\begin{equation}
		\frac{N_K}{K}= \frac{t_{2K}(\gG)}{K} 
		\underset{K\to\infty}\longrightarrow 
		\E\left[t_2(\gG)\right]=
		\E\left[L_\gG^\downarrow+L_\gG^\uparrow\right] \, .
	\end{equation}
	Since, by Lemma \ref{th:def_free_ener}, $\calM(J)$ is the almost sure limit of $M_{N_K, J, h}^{++}/N_K$, the corollary follows.
\end{proof}

\subsection{Statistics of the $\gG$-stretches for large $\gG$ and proof of Theorem \ref{th:first_order_calM}}

We now describe the statistics of the $\gG$-stretches when $\gG$ is large. In \cite{CGH25+} (see Proposition 1.3 there) it was shown using Donsker's invariance principle that, under the Brownian scaling, the heights and lenghts of generic $\gG$-stretches converge in law towards the corresponding variables for a Brownian motion, which have been studied in \cite{NP89}. Let us state this result.

\begin{fact}\label{th:fact_Donsker}
	Assume \eqref{eq:assump_mu_initial}. Then, as $\gG\to\infty$, the couples
	\begin{equation}
		\left(\frac{H^\downarrow_\gG-\gG}{\gG}, \frac{L_\gG^\downarrow}{\gG^2/\vartheta^2}\right) 
		\quad \text{ and } \quad 
		\left(\frac{H^\uparrow_\gG-\gG}{\gG}, \frac{L_\gG^\uparrow}{\gG^2/\vartheta^2}\right)
	\end{equation}
	both converge in distribution towards the same law on $[0, \infty)^2$, which is explicitely described in \cite{NP89}. This law is absolutely continuous with respect to the Lebesgue measure on $[0, \infty)^2$ and its marginals both have expectation 1.
\end{fact}

We deduce the following first order estimates.
\begin{corollary} \label{th:cor_Donsker} 
	Assume \eqref{eq:assump_mu_initial}. Then, as $\gG\to\infty$, 
	\begin{equation}
		\E\left[H^\downarrow_\gG\right]\sim \E\left[H^\uparrow_\gG\right] \sim 2\gG\, ,
	\end{equation}
	\begin{equation}
		\E\left[L^\downarrow_\gG\right]\sim \E\left[L^\uparrow_\gG\right] \sim \frac{\gG^2}{\vartheta^2}\, .
	\end{equation}
\end{corollary}

\begin{proof}
To derive the corollary from Fact \ref{th:fact_Donsker}, it is sufficient to prove that for large enough $\gG_0$, the families $\left(\frac{H^\downarrow_\gG}{\gG}\right)_{\gG>\gG_0}$, $\left(\frac{H^\uparrow_\gG}{\gG}\right)_{\gG>\gG_0}$, $\left(\frac{L_\gG^\downarrow}{\gG^2}\right)_{\gG>\gG_0}$ and $\left(\frac{L_\gG^\uparrow}{\gG^2}\right)_{\gG>\gG_0}$ are uniformly integrable.

We first treat the case of the lengths. Since $L_\gG^\downarrow + L_\gG^\uparrow$ is distributed as $t_2(\gG)$, it is enough to show that the families $\left(\frac{t_1(\gG)}{\gG^2}\right)_{\gG>\gG_0}$ and $\left(\frac{t_2(\gG)-t_1(\gG)}{\gG^2}\right)_{\gG>\gG_0}$ are unifomly integrable for large enough $\gG$. Those two families can be treated similarly, we focus on the first one. Using Fact \ref{th:fact_Donsker}, let us take $\gG_0$ large enough so that, for some large $C$, for every $\gG>\gG_0$, 
\begin{equation}
	\p\left[t_1(\gG) >   C \gG^2   \right] \le e^{-1}	
\end{equation} 
Denoting $L=\left\lfloor C \gG^2 \right\rfloor$, we bound the probability that there is no drop of size $\gG$ in $[\![0, nL]\!]$ by the probability that for all $k=0, \dots, n-1$ there is no drop of size $\gG$ in $[\![kL, (k+1)L]\!]$ to derive that for $\gG>\gG_0$ and every integer $n$:
\begin{equation}\label{eq:uniform_integr_t_1}
\bbP\left[ t_1(\gG)>n \left\lfloor C \gG^2 \right\rfloor \right] 
\le \bbP\left[ t_1(\gG) >  \left\lfloor C \gG^2 \right\rfloor \right]^n
\le e^{-n}
\end{equation}
Hence the family $\left(\frac{t_1(\gG)}{\gG^2}\right)_{\gG>\gG_0}$ is uniformly integrable.

Concerning the heights, let us focus on the family $\left(H_\gG^\downarrow\right)_{\gG > \gG_0}$. 
We are going to show that, as $x\to\infty$, 
\begin{equation}\label{eq:bound_unif_HgG}
	\sup_{\gG> \gG_0} \p[H_\gG^\downarrow \ge x \gG] = O\left(\frac{\log (x)}{x^2} \right)
\end{equation}
which readily implies uniform integrability of the family $\left(\frac{H^\downarrow_\gG}{\gG}\right)_{\gG>\gG_0}$.
To do so, we observe, by writing 
\begin{equation}
	\p[t_2(\gG) > 2 t ]\le \p[t_1(\gG) > t ]+ \p[t_2(\gG)-t_1(\gG) >t] \, , 
\end{equation}
that a statement similar to \eqref{eq:uniform_integr_t_1} holds with $t_1(\gG)$ replaced by $t_2(\gG)$ and we observe that if $H_\gG^\downarrow \ge x \gG$, and $t_2(\gG) \le \lfloor 3 \log(x) \rfloor \lfloor C\gG^2 \rfloor  $, then
\begin{equation}
	\sup_{0\le n \le \lfloor 3 \log(x) \rfloor \lfloor C\gG^2 \rfloor } |S_n| \ge \frac{x \gG}{2} \, .
\end{equation}
The probability of the latter event can be controlled using Kolmogorov's inequality. 
So we write:
\begin{equation}
	\begin{aligned}
		\p\left[H_\gG^\downarrow \ge x \gG\right] 
		& \le \p \left[\sup_{0\le n \le   \lfloor 3\log(x) \rfloor  \lfloor C\gG^2 \rfloor } |S_n| \ge \frac{x \gG}{2} \right] 
		+ \p\left[t_2(\gG) >  \lfloor 3\log(x) \rfloor  \lfloor C\gG^2 \rfloor \right]\\
		& \le \frac{  \lfloor 3\log(x) \rfloor \lfloor C\gG^2 \rfloor \vartheta^2}{\left(\frac{x \gG}{2}\right)^2} 
		+ \exp(-  \lfloor 3\log(x) \rfloor )  \\
		& \le 12 C \vartheta^2 \frac{ \log(x) }{x^2} + \frac{e}{x^3}\, .
	\end{aligned}
\end{equation}
This yields \eqref{eq:bound_unif_HgG} and the proof is complete.
\end{proof}

Theorem \ref{th:first_order_calM} readily follows from Corollary \ref{th:cor_expression_calM} and Corollary \ref{th:cor_Donsker}. 

\begin{remark}\label{rem:excess_is_large}
	The variables $H_\gG^\downarrow$ and $H_\gG^\uparrow$ are by definition not smaller than $\gG$. Fact \ref{th:fact_Donsker} states that the excess over $\gG$ is also of order $\gG$, which is a crucial feature for our bounds in Section \ref{sec:proof_main_thms} to be precise. Indeed, let us consider the configuration with maximal energy described in Fact \ref{th:fact_maximal_config}. When switching a whole domain of this configuration to opposite spin, not only does this always result in an energy loss, but most often in a \emph{large} energy loss. This advocates the idea (which we exploit rigorously in the proof of Theorem \ref{th:cor_dev_free_ener}, in Section \ref{sec:proof_main_thms}) that the configurations contributing the most to the partition function should in some sense comply with the process of $\gG$-extrema.
\end{remark}

\subsection{Environment around a $\gG$-extremum}\label{sec:around_extremum}

Let us introduce the following random piece of trajectory:
\begin{equation}
	 u_{-1}:=u_1(\gG) - u_2(\gG) < 0  \, , \qquad u_{+1}^+:= u_3^+(\gG)- u_2(\gG) >0 \, , 
\end{equation}
\begin{equation}
	 S^{(\gG)}
	=  (S^{(\gG)}_n)_{u_{-1}\le n \le u_{+1}^+ } 
	:=  (S_{u_2(\gG)+n}- S_{u_2(\gG)})_{u_{-1} \le n \le u_{+1}^+ } 
\end{equation}
In words, $S^{(\gG)}$ is the portion of trajectory around the $\gG$-minimum $u_2(\gG)$ consisting of the descending $\gG$-stretch on its left and the ascending $\gG$-stretch on its right, recentered around $u_2(\gG)$.

As in Section \ref{sec:intro_results_free_ener}, we extend $S$ to the whole of $\Z$ by first extending $h$ to an i.i.d. sequence $(h_n)_{n\in\Z}$ and setting:
	\begin{equation}
		S_0=0\, , \qquad 
		S_n = \sum_{i=1}^n h_i \text{ for } n\ge 1\, , \quad  
		S_n= - \sum_{i=n+1}^0 h_i \text{ for } n\le -1 \, .
	\end{equation}

\begin{remark}\label{rem:rotation_invariance}
	Using the classical observation that $(-S_{-n})_{n\in \Z}$, the rotation of  $(S_n)_{n\in \Z}$ around 0, is distributed as $(S_n)_{n\in \Z}$ itself, we obtain that the law of a descending $\gG$-stretch is invariant under rotation: $(S_{u_2(\gG)} -S_{u_2(\gG)  - n})_{0\le n \le u_2(\gG) - u_1(\gG)}$ is distributed as $(S_{u_2(\gG)+n}- S_{u_2(\gG)})_{u_{-1} \le n \le u_2(\gG) - u_1(\gG)}$. Similarly for an ascending $\gG$-stretch. 
	Consequently, the environment seen around a $\gG$-maximum (defined analogously as above) is distributed as the rotation of the environment seen around a $\gG$-minimum, i.e., as $\left(-S^{(\gG)}_{-n}\right)_{-u_{+1}^+ \le n \le - u_{-1}}$.
\end{remark}





Now, we define $(Z_n)_{n\in \Z}$, which is the limit of process $S^{(\gG)}$ as $\gG\to\infty$ in the sense of the finite-dimensional distributions. Let us describe it. $(Z_n)_{n\ge 0}$ is the random walk $(S_n)_{n\ge 0}$ started at 0 and conditioned not to visit $(-\infty, 0)$. A rigorous definition can be found in \cite{Bertoin1994}: it is a Doob's transform of $S$. 
Similarly, $(Z_n)_{n\le 0}$ is the random walk $(S_n)_{n\le 0}$ with $S_0=0$ and conditioned not to visit $(-\infty, 0]$ outside time 0. 
To define it rigorously, first build $(S_n)_{n\ge 0}$ started from 0 and conditioned not to visit  $[0, \infty)$ afterwards (using \cite{Bertoin1994} again), then apply a rotation around 0, so consider $(-S_{-n})_{n \le 0}$.

To prove the next lemma, we will have to control the convergence of $S^{(\gG)}$ towards $Z$ beyond the finite-dimensional distribution convergence. 
Before stating the lemma, we define two random times, with respect to $S^{(\gG)}$:
\begin{equation}
	\tau_{\gG/2}^- : = \sup\left\{u_{-1} \le n \le 0 : S^{(\gG)}_n \ge  \frac{\gG}{2}\right\}\, ,
\end{equation}
\begin{equation}
	\tau_{\gG/2}^+ : = \inf\left\{0 \le n \le u_{+1}^+ : S^{(\gG)}_n \ge  \frac{\gG}{2}\right\} \, .
\end{equation}

\begin{lemma}\label{th:lem_Elog_fine}
	Assume \eqref{eq:assump_mu_initial}. As $\gG\to\infty$, the expectations 
	\begin{equation}\label{eq:lem_Elog_fine_expectations}
	\E
	\left[
	\log
	\left(\sum_{\tau_{\Gamma/2}^- < n < \tau_{\Gamma/2}^+} e^{-2S^{(\gG)}_n} \right)
	\right]
	\qquad \text{and} \qquad 
	\E
	\left[
	\log
	\left(\sum_{u_{-1} \le n \le u^+_{+1}} e^{-2S^{(\gG)}_n} \right)
	\right]
	\end{equation}
	converge towards 
	\begin{equation}
	\E
	\left[\log
	\left(\sum_{n\in \Z} e^{-2Z_n} \right)
	\right] \, .
	\end{equation}
\end{lemma}

Before proving Lemma \ref{th:lem_Elog_fine}, let us state and prove the following lemma. For every $x$, under a probability measure denoted  $\p_x$, we consider the process $Z$ and the random walk $S$ to be started at $x$. In particular, $\p_0=\p$. If $A$ is a subset of $\R$, we denote 
\begin{equation}
	\tau_A=\inf\{n\ge 0: S_n\in A\}\, .
\end{equation}

\begin{lemma}\label{th:lem_tech_sums}
	Assume \eqref{eq:assump_mu_initial}. Then, there exists a constant $C$ such that, for every $x\ge 0$,  
	\begin{equation} \label{eq:lem_tech_sumS}
	\sup_{\gG\ge 1} \E_x\left[\sum_{n=0}^{ \tau_{[\gG, \infty)}-1} e^{-2S_n} \Big| \tau_{[\gG, \infty)}< \tau_{(-\infty, 0)}\right] \le \frac{C}{x+1} \, ,
\end{equation} 
and
	\begin{equation} \label{eq:lem_tech_sumZ}
	\E_x\left[\sum_{ n \ge 0} e^{-2Z_n}\right] \le \frac{C}{x+1} \, .
\end{equation}
	
\end{lemma}

\begin{proof}
	To begin with, let us handle the statement in the lemma concerning the walk $S$ conditioned on  the event $\{ \tau_{[\gG, \infty)}< \tau_{(-\infty, 0)} \}$. This event will be denoted $B_\gG$ throughout this proof. 
	
	We recall two useful ``gambler's ruin'' estimates. There exist positive constants $c_1, c_2$ and $c_3$ such that for every $0 \le x < r$ we have that
\begin{equation}\label{eq:Gambler_prob}
	c_1 \frac{x+1}{r} \le \p_x[\tau_{[r, \infty)}< \tau_{(-\infty, 0)}] \le c_2 \frac{x+1}{r}
\end{equation}
and:
\begin{equation}\label{eq:Gambler_time}
	\p_x[\tau_{[r, \infty)} \wedge \tau_{(-\infty, 0)} \ge r^2 ]\le c_3\frac{x+1}{r}\, .
\end{equation}

For a reference, see Propositions 5.1.4  and 5.1.5 in \cite{Lawler2010}, where these results are stated with $\tau_{(-\infty, 0)}$ replaced by $ \tau_{(-\infty, 0]}$: our statements follow by simple arguments. 

\smallskip

We start by establishing the existence of a positive constant $c_4$ such that for every $x\ge 0$, for every $\gG \ge 1$ and $y\ge 1$,
\begin{equation}\label{eq:bound_prob_hit_y}
	\p_x[\tau_{[0, y]} < \tau_{[\gG, \infty)}| B_\gG] \le c_4 \frac{y}{x+1}\, .
\end{equation}
Note that this statement is immediate if $x\ge \gG$ (then the probability is 0) or if $x\le y$ (then the probability is 1 but any $c_4\ge 2$ is sufficient). 
Assuming $1\le y< x < \gG$, we use the strong Markov property and \eqref{eq:Gambler_prob} to get:
\begin{equation}
	\begin{aligned}
		\p_x[\tau_{[0, y]} < \tau_{[\gG, \infty)}| B_\gG] 
		& = \frac{\p_x[\tau_{[0, y]} < \tau_{[\gG, \infty)} < \tau_{(-\infty, 0)}]}{\p_x[\tau_{[\gG, \infty)} < \tau_{(-\infty, 0)}]}\\
		& \le \frac{\sup_{z\in[0, y]} \p_z[\tau_{[\gG, \infty)} < \tau_{(-\infty, 0)}]}{\p_x[\tau_{[\gG, \infty)} < \tau_{(-\infty, 0)}]}\\
		& \le \frac{c_2 \frac{y+1}{\gG}}{c_1 \frac{x+1}{\gG}} \le \frac{2c_2}{c_1} \frac{y}{x+1}\, .
	\end{aligned}
\end{equation}

\smallskip 

We are further going to establish a control on the expected time spent below $y$ by the walk conditioned on $B_\gG$. Let us first show that there exists $c_5>0$ such that for every $1\le y\le \gG$ and every $z\in [0, y)$,
\begin{equation}\label{eq:bound_time_tau_y}
	\E_z[\tau_{[y, \infty)}| B_\gG] \le c_5 y^2 \, .
\end{equation}
To prove this we observe, for any $t \ge 0$, that if $B_\gG$ holds and $\tau_{[y, \infty)} > t $, then $S$ remains inside $[0, y)$ at least up to time $\lfloor t \rfloor$ — in particular it is inside  $[0, y)$ at that time — and after that time it hits $[\gG, \infty)$ before hitting $(-\infty, 0)$, hence:
\begin{equation}
	\begin{aligned}
	\p_z[\tau_{[y, \infty)} > t| B_\gG] 
	& \le  
	\frac{
		\p_z[ \tau_{[y, \infty)} \wedge \tau_{(-\infty, 0)}  > t ] 
		\sup_{z'\in[0, y)} \p_{z'}[\tau_{[\gG, \infty)} < \tau_{(-\infty, 0)}]}
	{\p_z[\tau_{[\gG, \infty)} < \tau_{(-\infty, 0)}] }\\
	& \le 
	\frac{\p_z[  \tau_{[y, \infty)} \wedge \tau_{(-\infty, 0)}  > t ]  
	c_2 \frac{y+1}{\gG}}
	{c_1 \frac{z+1}{\gG}} \\
	& \le \p_z[ \tau_{[y, \infty)} \wedge \tau_{(-\infty, 0)} > t ]
	\frac{2c_2}{c_1} \frac{y}{z+1}  \\
	\end{aligned}
\end{equation}
Let us now pick $c_5'$ such that $c_5'\ge 1$ and $  \frac{c_3}{\sqrt{c_5'}}  \frac{2c_2}{c_1}\le \frac{1}{2}$.
Then we have, using \eqref{eq:Gambler_time},
\begin{equation}
	\begin{aligned}
	\p_z[ \tau_{[y, \infty)} \wedge \tau_{(-\infty, 0)} > c_5' y^2 ] 
	& \le \p_z[  \tau_{[\sqrt{c_5'} y, \infty)} \wedge \tau_{(-\infty, 0)} \ge c_6' y^2 ] \\
	& \le c_3 \frac{z+1}{\sqrt{c_5'} y}\, .
	\end{aligned}
\end{equation}
So, for every $1\le y \le \gG$, we have:
\begin{equation}
	\sup_{z\in [0, y)} \p_z[\tau_{[y, \infty)} > c_5' y^2| B_\gG] \le \frac{1}{2} \, ,
\end{equation}
and by an induction argument we deduce that for every $1\le y \le \gG$, for every $n\ge1$, 
\begin{equation}
	\sup_{z\in [0, y)} \p_z[\tau_{[y, \infty)} > n \lfloor c_5' y^2 \rfloor |  B_\gG] \le \frac{1}{2^n}
\end{equation}
and \eqref{eq:bound_time_tau_y} follows.
As annouced we derive from \eqref{eq:bound_prob_hit_y} and \eqref{eq:bound_time_tau_y} a control on the time spent in $[0, y)$ up to time $\tau_{[\gG, \infty)}$ by the walk conditioned on $B_\gG$. There exists $c_6>0$ such that for every $x\ge 0$, for every $\gG\ge 1$ and $y \ge 1$, 
\begin{equation}\label{eq:bound_expected_time_in_0y}
	\E_x\left[\sum_{n=0}^{\tau_{[\gG, \infty)}-1 } \ind_{S_n \in [0, y]} |  B_\gG\right]  \le c_6 \frac{y^3}{x+1} \, .
\end{equation}
Let us prove this. Without loss of generality, we may assume that $x<\gG$, otherwise the sum is void. If $y\ge \frac{\gG}{2c_4}$ then we bound the sum by $\tau_{[\gG, \infty)}$ and use \eqref{eq:bound_time_tau_y} with $y=\gG$.
Otherwise we reason as follows. Starting from $x$, the conditioned walk has a probability bounded by $c_4\frac{y}{x+1}$ of hitting $[0, y]$; if this happens, then, in virtue of \eqref{eq:bound_time_tau_y}, the conditioned walk remains in $[0, 2c_4 y)$ for an expected time bounded by a constant times $y^2$; after leaving this segment the conditioned walk has probability bounded by $c_4\frac{ y}{2c_4 y+1} \le \frac{1}{2}$ of returning to $[0, y]$ and so on; the number of returns to $[0, y]$ is stochastically dominated by a geometric variable with parameter 2.
The bound \eqref{eq:bound_expected_time_in_0y} follows.

\medskip

Finally, we establish \eqref{eq:lem_tech_sumS} by writing:
\begin{equation}
	\begin{aligned}
	\E_x  \left[\sum_{0 \le n \le \tau_{[\gG, \infty)}-1} e^{-2S_n} | B_\gG\right] 
	& \le \sum_{p=1}^{\infty} e^{-2(p-1)}  \E_x\left[\sum_{0 \le n \le \tau_{[\gG, \infty)}-1} \ind_{S_n \in [p-1, p]}  | B_\gG \right] \\
	& \le \sum_{p=1}^{\infty} e^{-2(p-1)}  \E_x\left[\sum_{0 \le n \le \tau_{[\gG, \infty)}-1} \ind_{S_n \in [0, p]}  | B_\gG \right] \\
	& \le \frac{c_6}{x+1}\sum_{p=1}^{\infty} e^{-2(p-1)} p^3  \, ,
	\end{aligned}
\end{equation}
where we used \eqref{eq:bound_expected_time_in_0y} for the last inequality.

\bigskip

Let us turn to the statement in the lemma concerning the limiting process $(Z_n)_{n\ge 0}$. We are going to derive it from \eqref{eq:lem_tech_sumS} but we need to be careful since the convergence of $S$ under $\p_x[ \cdot | B_\gG]$ towards $Z$ is only in the sense of finite-dimensional distributions. Let us fix $n\ge 0$, and use the fact that the law of $(S_0, \dots, S_n)$ under $\p_x[\cdot | B_\gG]$ converges weakly towards the law of $(Z_1, \dots, Z_n)$ under $\p_x$. We claim that 
\begin{equation}\label{eq:prob_n<tau_to_1}
	\p_x[ n <\tau_{[\gG, \infty)} |B_\gG]\underset{\gG\to\infty}{\longrightarrow} 1\, .
\end{equation}
Indeed for any $\gamma>0$, as soon as $\gG \ge \gamma$, 
\begin{equation}
	\p_x[ n< \tau_{[\gG, \infty)}  |B_\gG] \ge  \p_x[\tau_{[\gamma, \infty)} > n |B_\gG] = \p_x[(S_0, \dots, S_n) \in (-\infty, \gamma)^{m+1}  |B_\gG] 
\end{equation}
and using the Portemanteau theorem, the infimum of the latter, as $\gG\to\infty$, is not smaller than $\p_x[(Z_0, \dots, Z_n) \in (-\infty, \gamma)^{m+1}] $, finally this probability converges to 1 as $\gamma \to\infty$. Since the function $x \mapsto e^{-2x}$ is continuous and bounded on $[0, \infty)$, we deduce that
\begin{equation}
	\E_x  \left[  e^{-2S_n}  \ind_{n< \tau_{[\gG, \infty)}} |B_\gG \right] \underset{\gG\to\infty}{\longrightarrow} \E_x  \left[  e^{-2Z_n} \right]
\end{equation}
From this, together with \eqref{eq:lem_tech_sumS}, we deduce that for each $m\ge0$,
\begin{equation}
	\begin{aligned}
	\E_x  \left[ \sum_{0\le n\le m}  e^{-2Z_n} \right] 
	& = \lim_{\gG\to\infty} \sum_{0\le n \le m } \E_x  \left[  e^{-2S_n}  \ind_{n< \tau_{[\gG, \infty)}} |B_\gG \right] \\
	& \le \lim_{\gG\to\infty} \sum_{n \ge 0} \E_x  \left[  e^{-2S_n}  \ind_{n< \tau_{[\gG, \infty)}} |B_\gG \right] \le \frac{C}{x+1}\, .
	\end{aligned}
\end{equation}
Using the monotone convergence theorem, we take $m\to\infty$ and deduce \eqref{eq:lem_tech_sumZ}.
\end{proof}

With this lemma at hand, we prove Lemma \ref{th:lem_Elog_fine}.

\begin{proof}[Proof of Lemma \ref{th:lem_Elog_fine}]
Our first goal is to prove that the variables 
\begin{equation}\label{eq:desired_conv_full}
	\sum_{\tau_{\Gamma/2}^- < n < \tau_{\Gamma/2}^+} e^{-2S^{(\gG)}_n} 
	\qquad \text{and} \qquad 
	\sum_{u_{-1} \le n \le u^+_{+1}} e^{-2S^{(\gG)}_n} 
\end{equation}
converge towards the variable $\sum_{n\in\Z}  e^{-2Z_n}$.
In fact, we are going to focus on proving the convergence of the variables  
\begin{equation}\label{eq:desired_conv_right_piece}
	\sum_{0 \le  n < \tau_{\Gamma/2}^+} e^{-2S^{(\gG)}_n} 
	\qquad \text{and} \qquad 
	\sum_{0 \le n \le u^+_{+1}} e^{-2S^{(\gG)}_n} 
\end{equation}
towards the variable $\sum_{n\ge 0}  e^{-2Z_n}$, since the rest of the sums can be treated similarly.

Let us observe that the trajectory $(S^{(\gG)}_n)_{0 \le n \le u^+_{+1}} $ can be decomposed into two pieces:
\begin{itemize}
\item from time 0 to time $\tau_\gG^+= \inf \{n \ge 0: S^{(\gG)}_n \ge \gG\}$, the trajectory is distributed as $S$ started from 0 and conditioned on hitting $[\gG, \infty)$ before hitting $(-\infty, 0)$, and stopped at the time it hits $[\gG, \infty)$. 
\item from time $\tau_\gG^+$ to time $u_{+1}^+$, the trajectory is distributed as $S$ stopped at the beginning of the first downward excursion with depth $\gG$ or larger.
\end{itemize}
By \cite{Bertoin1994}, we have that for each $m\ge 1$, $(S_0^{(\gG)}, \dots, S^{(\gG)}_{m-1})$ converges in distribution towards $(Z_0, \dots, Z_{m-1})$. 
Reasoning as for \eqref{eq:prob_n<tau_to_1}, we also note that 
\begin{equation}
\p[\tau_{\gG/2}^+ >n]
= \p[\tau_{[\gG/2, \infty)} >n | \tau_{[\gG, \infty)}< \tau_{(-\infty, 0)}] \underset{\gG\to\infty}{\longrightarrow} 1\, .
\end{equation}
To deduce the convergences we are after, it is enough to show a bound on 
$\E\left[\sum_{n \ge m} e^{-2Z_n}\right]$ 
which vanishes as $m\to\infty$ and a bound on 
$\E\left[\sum_{m\le n \le u_{+1}} e^{-2S_n^{(\gG)}}\right]$ 
which is uniform in $\gG$ and vanishes as $m\to\infty$. 
Furthermore, since $S_m^{(\gG)}$ converges in distribution towards $Z_m$ and since $Z_m\to\infty$ as $m\to\infty$, it is enough to show a bound on 
$\E_x\left[\sum_{n \ge 0} e^{-2Z_n}\right]$ 
which vanishes as $x\to\infty$ and a bound on 
$\E_x\left[\sum_{0\le n \le u_{+1}} e^{-2S_n^{(\gG)}}\right]$ which is uniform in $\gG\ge 1$ and vanishes as $x\to\infty$, where, under $\p_x$, we consider $S^{(\gG)}$ to be as describe above, but started from $x$.

The bound on 
$\E_x\left[\sum_{n \ge 0} e^{-2Z_n}\right]$ 
is provided by Lemma \ref{th:lem_tech_sums}, which also provides a bound for the first piece of $S^{(\gG)}$ (that is, up to time $\tau_\gG^+$). For the second piece we use some excursion theory. We use the formalism introduced in the proof of Fact \ref{th:fact_indep_portions}, so the weakly ascending record times:
\begin{equation}
\alpha_0:=0 \, , \quad 
\text{and, for } k\ge 1, \quad 
\alpha_k:=\min\{n\ge \alpha_{k-1}: S_n \ge S_{\alpha_{k-1}}\} \, , 
\end{equation} 
and further set:
\begin{itemize}
	\item $\Delta_k=S_{\alpha_k}-S_{\alpha_{k-1}}$, the overshoot of the $k$-th excursion; 
	\item $D_k=\max_{\alpha_{k-1} \le n < \alpha_k } S_{\alpha_{k-1}} - S_n$, the depth of the $k$-th excursion;
	\item $\eta_k=\sum_{n=\alpha_{k-1}}^{\alpha_k - 1} e^{-2(S_n-S_{\alpha_{k-1}})}$, the contribution of the $k$-th excursion;
\end{itemize}
so that, remarking that  $S^{(\gG)}$ is at $\gG$ or above at time $\tau_{\gG}^+$,
\begin{equation}
\E_x\left[\sum_{\tau_{\gG}^+ \le n \le u_{+1}^+} e^{-2S_n^{(\gG)}}\right] \le e^{-2\gG} \E\left[\sum_{k\ge 1} e^{-2\sum_{l=1}^{k-1} \Delta_k } \eta_k \ind_{\forall l=1, \dots, k, D_l < \gG}\right]
\end{equation}
(we recall that under $\p$, $S$ is started at 0).
Now, let us simply bound the indicator function $\ind_{\forall l=1, \dots, k, D_l < \gG}$ by $\ind_{D_k < \gG}$, and use the independence of the excursions:
\begin{equation}
	\begin{aligned}
	\E_x\left[\sum_{\tau_{\gG}^+ \le n \le u_{+1}^+} e^{-2S_n^{(\gG)}}\right] 
	& \le e^{-2\gG} \E\left[\sum_{k\ge 1} e^{-2\sum_{l=1}^{k-1} \Delta_k } \eta_k \ind_{ D_k < \gG}\right]\\
	& \le e^{-2\gG} \sum_{k\ge 1} \E\left[e^{-2 \Delta_1 }\right]^{k-1} \E[\eta_1 \ind_{ D_1 < \gG}]\\
	& = e^{-2\gG} \frac{1}{1-\E\left[e^{-2 \Delta_1 }\right]} \E[\eta_1 \ind_{ D_1 < \gG}]\, .
	\end{aligned}
\end{equation}
For $A\subset \R$, let us denote $\tau^+_A := \inf\{n\ge 1 : S_n \in A\}$. We observe that:
\begin{equation}
	\begin{aligned}
	e^{-2\gG}  \E[\eta_1 \ind_{ D_1 < \gG}] 
	&  \le  e^{-2\gG}  \E[\eta_1 | D_1 < \gG]\\
	& = e^{-2\gG} \E\left[\sum_{0\le n \le \tau^+_{[0, \infty)}-1} e^{-2S_n} 
	|  \tau^+_{[0, \infty)} < \tau_{(-\infty, \gG]}\right] \\
	&  = \E_{\gG}\left[\sum_{0\le n \le \tau^+_{[\gG, \infty)}-1} e^{-2S_n}  
	|  \tau^+_{[\gG, \infty)} < \tau_{(-\infty, 0]} \right] \, .
	\end{aligned}
\end{equation} 
This quantity is very similar to the quantity in Lemma \ref{th:lem_tech_sums} when taking $x=\gG$ there, except that  $\tau_{[\gG, \infty)}$  is replaced by $\tau^+_{[\gG, \infty)}$. In fact, loosely speaking, sending $x\to\gG$ from below in \eqref{eq:lem_tech_sumS} should yield that:
\begin{equation}
	\E_{\gG}\left[\sum_{0\le n \le \tau^+_{[\gG, \infty)}-1} e^{-2S_n}  
	|  \tau^+_{[\gG, \infty)} < \tau_{(-\infty, 0]} \right] \le \frac{C}{\gG+1}\, .
\end{equation} 
Actually, this is not legitimate, because we do not control continuity of the walk conditioned on $B_\gG$ in the initial point, but one can go through the proof of Lemma \ref{th:lem_tech_sums} and adapt it to prove this claim. 

\bigskip

We thus derive the convergences stated in \eqref{eq:desired_conv_right_piece}. We proceed similarly for the piece of trajectory on the left of 0 and deduce the convergences stated in \eqref{eq:desired_conv_full}.

\medskip

While showing these convergences, we have also shown that the variables in \eqref{eq:desired_conv_full} have uniformly bounded expectations, for $\gG\ge 1$; using the fact that $\log(u)=o(u)$ as $u\to\infty$, we derive that the variables in Lemma \ref{th:lem_Elog_fine} are uniformly integrable and this concludes the proof of the lemma.
\end{proof}

\section{Proof of Theorem \ref{th:dev_calF}} \label{sec:proof_main_thms}

\subsection{Heuristic}

We give here a high-level overview of the main ideas behind the proof of Theorem~\ref{th:dev_calF}.

\medskip

For the \emph{lower bound}, we focus on configurations that deviate slightly from the maximal energy configuration described in Fact~\ref{th:fact_maximal_config}. Specifically, we allow for \emph{small displacements} of the domain walls in that configuration. These local modifications yield a quantifiable energy cost, which results in the additive correction term \(\frac{\widetilde{\kappa}}{(2J)^2}\) appearing in the lower bound of Theorem~\ref{th:dev_calF}.


\medskip

For the \emph{upper bound}, we base our analysis on the process of \((\gG - 5 \log \gG)\)-extrema of the walk \(S\). Let us denote by \(\sigma^{(\gG - 5 \log \gG)}\) the configuration prescribed (in the sense of Fact~\ref{th:fact_maximal_config}) by these \((\gG - 5 \log \gG)\)-extrema. The proof proceeds in several steps:


\begin{itemize}

\item {\bf Step 1.} In the same spirit as in the proof of Fact~\ref{th:fact_maximal_config}, we show that any configuration \(\sigma\) that introduces ``too many'' domain walls compared to \(\sigma^{(\gG - 5 \log \gG)}\) is suboptimal. Moreover, the gap between \(\gG - 5 \log \gG\) and \(\gG\) allows for a quantitative estimate of the energy penalty incurred by such excess walls. This justifies restricting our attention to configurations obtained from \(\sigma^{(\gG - 5 \log \gG)}\) by performing \emph{small displacements} of domain walls, or by removing some of them.


\item {\bf Step 2.} We then quantify the energy cost of these small displacements. After accounting for this contribution, we can further restrict to configurations in which spin changes occur only at the positions of the \((\gG - 5 \log \gG)\)-extrema.


\item {\bf Step 3.} We observe that any \((\gG - 5 \log \gG)\)-stretch with height exceeding \(\gG + 5 \log \gG\) must contain a corresponding spin domain in any nearly-optimal configuration. Once again, we can provide a quantitative estimate for the energy gain from including such domains. This leaves only a small number of remaining possibilities — namely, the configuration remains free on the stretches whose height does not exceed \(\gG + 5 \log \gG\), which form a small fraction of the total. For each such residual configuration, we bound its energy from above by the maximal energy.

\end{itemize}
Finally, we combine the estimates obtained in the previous steps, and invoke the LLN together with our controls from Fact~\ref{th:fact_Donsker}, Corollary~\ref{th:cor_Donsker} and Lemma~\ref{th:lem_Elog_fine}. In particular, the additive term \(\frac{\widetilde{\kappa}}{(2J)^2}\) in the upper bound arises from small displacements of domain walls around the \((\gG - 5 \log \gG)\)-extrema, in the same way as in the lower bound.

\subsection{Lower bound}

In this section, for convenience we are going to drop the dependence in $\gG$ in the notations, so we write $t_k$, $u_k$ and $u_k^+$ instead of $t_k(\gG)$, $u_k(\gG)$ and $u_k^+(\gG)$. 
We work with some large integer $K$ and consider the system of length $N=N_K=t_{2K}$. For every $k=2, \dots, 2K$, we define two times $\tau_k^-$ and $\tau_k^+$ around $u_k$, in the following manner. If $k$ is even, $u_k$ is a $\gG$-minimum and we denote
\begin{equation}
	\tau_k^-=\sup\{n<u_k: S_n\ge S_{u_k}+\gG/2\}\, , \qquad \tau_k^+=\inf\{n>u_k: S_n\ge S_{u_k}+\gG/2\}\, .
\end{equation}
If instead $k$ is odd, $u_k$ is a $\gG$-maximum and we denote
\begin{equation}
	\tau_k^-=\sup\{n<u_k: S_n\le S_{u_k}-\gG/2\}\, , \qquad \tau_k^+=\inf\{n>u_k: S_n\le S_{u_k}-\gG/2\}\, .
\end{equation}
We observe that 
\begin{equation}
	u_1 < \tau_2^-  +1 \, , \qquad \text{for every } k, \tau_k^+-1 < \tau_{k+1}^-+1\, , \text{ and } \tau_{2K}^+ \le  N.
\end{equation}
Now we consider the configurations $\sigma$ obtained by choosing a collection $(n_k)_{k=2, \dots, 2K}$ of times such that $\tau_k^- < n_k < \tau_k^+  $ for every $k$ and by setting $\sigma_0=0$ and flipping spins at position $u_1$ and at the positions $(n_k)_{k=2, \dots, 2K}$. Such a configuration has the energy:
\begin{equation}
	H_{N, J,h}^{++}(\sigma)=M_{N, J,h}^{++} - \sum_{k=2}^{2K} 2 (-1)^k (S_{n_k}-S_{u_k})
\end{equation}
Summing over these configurations yields:
\begin{equation}
Z_{N, J, h}^{++} \ge 
\exp(M_{N, J,h}^{++} ) 
\prod_{k=2}^{2K} 
\left(\sum_{\tau_k^- < n_k < \tau_k^+} \exp(-2 (-1)^k (S_{n_k}-S_{u_k})) \right)
\end{equation}
Taking the $\log$ and dividing by $N$:
\begin{equation}
\frac{1}{N} \log \left(Z_{N, J, h}^{++} \right) \ge 
\frac{M_{N, J,h}^{++}}{N} 
+\frac{2K}{N}\times \frac{1}{2K} \sum_{k=2}^{2K} 
\log \left(\sum_{\tau_k^- < n_k < \tau_k^+} \exp(-2 (-1)^k (S_{n_k}-S_{u_k})) \right)
\end{equation}
As $K\to\infty$, $N=t_{2K}\to\infty$ so, by Lemma \ref{th:def_free_ener}, the left-hand side of the inequality converges almost surely towards $\calF(J)$ and the first term on the right-hand side converges towards $\calM(J)$. Furthermore, using twice the LLN (together with Remarks \ref{rem:LLN} and \ref{rem:rotation_invariance}) the second addendum converges towards $\frac{2}{\E[L^\downarrow_\gG+L^\uparrow_\gG]}$ multiplied by 
\begin{equation}
\E\left[ \log \sum_{\tau_{\gG/2}^- < n < \tau_{\gG/2}^+} \exp\left(-2S^{(\gG)}_n\right)\right]\, ,
\end{equation} 
where we use the same notations as in Lemma \ref{th:lem_Elog_fine}.
We finally use Corollary \ref{th:cor_Donsker} and Lemma \ref{th:lem_Elog_fine} to derive the lower bound in Theorem \ref{th:dev_calF}.

\subsection{Upper bound}\label{sec:proof_upp_bound}

We are going to proceed in different steps. Instead of considering the process of $\gG$-extrema, we set $c_\gG=5\log \gG$ and consider the process of $(\gG-c_\gG)$-extrema.  Here again, for convenience we are going to drop the dependence in $\gG-c_\gG$ in the notations, so we write $t_k$, $u_k$ and $u_k^+$ instead of $t_k(\gG-c_\gG)$, $u_k(\gG-c_\gG)$ and $u_k^+(\gG-c_\gG)$. 
As before, we consider large integer $K$ and we are going to bound (this time from above) the partition function on the system of size $N=N_K=t_{2K}$, but we insist on the fact that now $t_{2K}$ is defined with $\gG-c_\gG$ instead of $\gG$.
For pratical purposes, we also set $u_0=0$ and $u_{2K+1}^+=t_{2K}$.


\subsubsection*{First step.}
By construction, inside an ascending (respectively, descending) $(\gG-c_\gG)$-stretch, there is no drop (respectively, rise) larger than $\gG-c_\gG$. 
Using this fact, we reduce the summation defining $Z_{N, J, h}^{++}$ to those configurations satisfying the following conditions:
\begin{enumerate}
\item For every $k=1, \dots, 2K$ there exists at most one wall located at a position between $u_{k-1}$ and $u_{k+1}^+$ (inclusive) having the same type as $u_k$ (i.e., a wall from +1 to -1 for $k$ odd, a wall from $-1$ to $+1$ for $k$ even). If such a wall exists, we refer to it as the wall attached to $u_k$.
\item For every $k=1, \dots, 2K-1$, provided they exist, the wall attached to $u_k$ is on the left of the wall attached to $u_{k+1}$.
\end{enumerate}
To justify this restriction, we make the following observation. If $\sigma$ is a configuration which does not follow those conditions, then it is always possible to remove a pair of walls in such a way that the energy increases by at least $2c_\gG$.
Indeed:
\begin{itemize}
\item Assume that the first condition is violated, say for some odd $k$ (so $u_k$ is a $(\gG-c_\gG)$-maximum); then the configuration $\sigma$ possesses a - domain entirely contained in $\lbra u_{k-1}, u_k\rbra$ or a + domain entirely  contained in $\lbra u_k, u_{k+1}\rbra$; in the first case (respectively, the second case) this domain corresponds to a drop (respectively a rise) in $S$ of size smaller than $\gG-c_\gG$ and removing this domain increases the energy by at least $2\gG- 2(\gG-c_\gG) = 2c_\gG$.
\item Assume that the first condition is satisfied but the second condition is violated, say for some odd $k$ (so $\lbra u_k, u_{k+1}\rbra$ is a descending $\gG$-stretch); then, similarly as before, the two walls correspond to a rise in $S$ of size smaller than $\gG-c_\gG$ and removing those two walls increases the energy by at least $2\gG- 2(\gG-c_\gG) = 2c_\gG$.
\end{itemize}
So, starting from some configuration $\sigma$ let us remove pairs of walls as described above until we obtain a configuration meeting the conditions and let us denote by $r$ the number of steps needed, so the number of removed pairs of walls. There may be a number of choices to be made along the procedure, let us assume we fix a deterministic rule. We claim that there are at most $\binom{N}{2r}$ configurations that lead in $r$ steps to a given configuration satisfying the conditions. 
Denoting by $C_{N, J, h}^{++}$ the sum defining $Z_{N, J, h}^{++}$ but restricted to configurations satisfying the conditions, it follows that:
\begin{eqnarray}
	Z_{N, J, h}^{++} \le C_{N, J, h}^{++} \times \sum_{r\ge 0} \binom{N}{2r} \exp(-2rc_\gG) = C_{N, J, h}^{++} \times (1+e^{-2c_\gG})^N.
\end{eqnarray}

\subsubsection*{Second step.}

We are going to relate $C_{N, J, h}^{++}$ to the partition function of an auxiliary RFIC.

Let us consider a configuration $\sigma$ satisfying the conditions. To $\sigma$ we associate the only configuration $\eta:\{2, \dots, 2K\}\to \{-1, +1\}$ such that, with respect to the boundary conditions $\eta_1=+1$ and $\eta_{2K+1}=+1$, for each $k\in\{1,\dots,  2K\}$, $\eta_{k}\neq \eta_{k+1}$ if and only if there is a wall attached to $u_{k}$ in configuration $\sigma$. In other terms, for each $k$, $\eta_{k}$ reflects the value taken by the configuration $\sigma$ on the $(\gG-c_\gG)$-stretch from $u_{k-1}$ to $u_{k}$, more precisely inbetween the (possible) walls attached to $u_{k-1}$ and $u_{k}$. 
Then, the energy of configuration $\sigma$ can be rewritten as:
\begin{equation}
	H_{N, J, h}^{++}(\sigma)
	= -\gG \sum_{k=1}^{2K} \ind_{\eta_k\neq \eta_{k+1}} 
	+ S_{u_1} + \sum_{k=2}^{2K} H_k\eta_k 
	+ (S_{t_{2k}}-S_{u_{2K}})
	+ \sum_{k} 2(-1)^k (S_{n_k} -S_{u_k})\, , 
\end{equation}
where the sequence $H=(H_k)_{k\ge 2}$ is defined by $H_k=S_{u_{k}}-S_{u_{k-1}}$. 
Conversely, we observe that the configurations that are associated to a given $\eta$ are given by the positions of the walls, so a collection of the following form: for each $k\in \{1, \dots, 2K\}$ such that $\eta_{k-1}\neq \eta_{k}$ a value $n_k$ between $u_{k-1}$ and $u_{k+1}^+$ (inclusive), with the further requirement that the sequence of $n_k$'s must be increasing. 

The quantity $C_{N, J, h}^{++}$ can thus be rewritten as:
\begin{equation}
	\begin{aligned}
	C_{N, J, h}^{++}
	= \sum_{\eta\in\{+1, -1\}^{2K}} 
	\Bigg( & 
		\exp\left(-\gG \sum_{k=1}^{2K} \ind_{\eta_k\neq \eta_{k+1}} 
	+\sum_{k=2}^{2K} H_k\eta_k + (S_{u_1}+S_{t_{2k}}-S_{u_{2K}})\right) \\
	& \sum_{(n_k)_k \text{ corresponding to } \eta} \exp\left( \sum_{k} 2(-1)^k (S_{t_k} -S_{u_k}) \right)
	\Bigg)
	\end{aligned}
\end{equation}
We now formally extend the sum over $(n_k)_k$ corresponding to $\eta$ to a sum over all collections $(n_k)_{k=1, \dots, 2K}$ such that $u_{k-1}\le n_k \le u_{k+1}^+$ for each $k$ to yield the following upper bound:
\begin{equation}
	\begin{aligned}
	C_{N, J, h}^{++}
	\le & \, Z_{1,2K, J, H}^{++} \exp( S_{u_1}+S_{t_{2k}}-S_{u_{2K}})\\
	& \times \prod_{k=1}^{2K} \left( 
		\sum_{u_{k-1} \le n_k \le  u_{k+1}^+} \exp\left(2(-1)^k (S_{n_k} -S_{u_k}) \right)
		\right) \, ,
	\end{aligned}
\end{equation}
where 
$Z_{1,2K, J, H}^{++}
=\sum_{\eta:\{2, \dots, 2K\}\to \{+1, -1\}^{2K}} 
\exp\left(-\gG \sum_{k=1}^{2K} \ind_{\eta_k\neq \eta_{k+1}} 
+\sum_{k=2}^{2K} H_k\eta_k\right)
$
is the partition of function of the RFIC  on the system $\{2, \dots, 2K\}$, with external field $H=(H_k)_{k\ge 2}$.


\subsubsection*{Third step.}

We now focus on $Z_{1,2K, J, H}^{++}$, which is the partition function associated to an external field with values that alternate signs and are large: $|H_k| \ge \gG-c_\gG$ by definition of $H_k$. Having in mind Fact \ref{th:cor_Donsker} and Remark \ref{rem:excess_is_large}, the $|H_k|$'s are in fact typically larger than, say, $\gG+c_\gG$ and this auxiliary chain should tend to align with the external field. Let us make this more rigorous.

Let us denote by $I_K$ the indices $k \in \{2, \dots, 2K\}$ such that $H_k\in[\gG-c_\gG, \gG+c_\gG]$. Let us consider a configuration $\eta\in\{+1, -1\}^{\{2, \dots, 2K\}}$. For every $k\in\{2, \dots, 2K\}\setminus I_K$, if $\eta_k$ is not aligned with $H_k$, flipping $\eta_k$ induces an energy gain of at least $2c_\gG$. 
Doing so we end with a configuration which is such that $\eta_k$ is aligned with $H_k$ for each $k\in\{2, \dots, 2K\}\setminus I_K$. We observe that there are at most $2^{\#I_K}$ such configurations; for which we simply bound the energy by the maximal energy
\begin{equation}
M_{1,2K, J, H}^{++}
:=\sup_{\eta:\{2, \dots, 2K\}\to \{+1, -1\}^{2K}} -\gG 
\sum_{k=1}^{2K} \ind_{\eta_k\neq \eta_{k+1}} 
+\sum_{k=2}^{2K} H_k\eta_k \, .
\end{equation}
It follows that:
\begin{equation}
	\begin{aligned}
	Z_{2K, J, H}^{++}
	& \le \exp(M_{1,2K, J, H}^{++}) \times  2^{\# I_K} \times \sum_{i\ge 0} \binom{2K-1}{i} \exp(-2 i c_\gG)\\
	& = \exp(M_{1,2K, J, H}^{++}) \times  2^{\# I_K}   \times (1+e^{-2c_\gG})^{2K-1}\, .
	\end{aligned}
\end{equation}
Let us finally observe that $M_{1,2K, J, H}^{++} + S_{u_1}+S_{t_{2k}}-S_{u_{2K}}$ can be intrepeted as the maximum defining $M_{N, J, h}^{++}$, restricted to a smaller set: the set of configurations having walls at positions of $(\gG-c_\gG)$-extrema. Hence, $M_{1,2K, J, H}^{++} + S_{u_1}+S_{t_{2k}}-S_{u_{2K}}$ is not larger than $M_{N, J, h}^{++}$ (in fact, they are equal, due to Fact \ref{th:fact_maximal_config}).

\subsubsection*{Conclusion.}
Combining the estimates obtained in the previous steps, taking the $\log$ and dividing by $N=N_K=t_{2K}$ (and using the fact that $2K-1\le N$), we get:
\begin{equation}
	\begin{aligned}
		\frac{1}{N} \log\left(Z_{N, J, h}^{++}\right) \le  & \frac{M_{N, J, h}^{++}}{N} + 2 \log\left(1+e^{-2c_\gG}\right)  
		\\
		& + \frac{1}{N}
		\sum_{k=1}^{2K} \log\left( 
		\sum_{u_{k-1} \le n_k \le u_{k+1}^+} \exp\left(2(-1)^k (S_{t_k} -S_{u_k}) \right)
		\right)\\
		& + \frac{\#I_K}{N}\log 2
	\end{aligned}
\end{equation}
Let us now take $K$ to infinity. Then $N$ goes to infinity and, by Lemma \ref{th:def_free_ener}, $\frac{1}{N} \log\left(Z_{N, J, h}^{++}\right) $ converges almost surely towards $\calF(J)$, while $\frac{M_{N, J, h}^{++}}{N}$ converges almost surely towards $\calM(J)$. Using the LLN, almost surely,
\begin{equation}
\frac{N}{2K}=\frac{t_{2K}}{2K}\underset{K\to\infty}{\longrightarrow}
\frac{\E[t_2(\gG - c_\gG)]}{2} = \frac{\E[L_{\gG-c_\gG}^\downarrow+L_{\gG-c_\gG}^\uparrow]}{2}\, ,
\end{equation}
while (using also Remarks \ref{rem:LLN} and \ref{rem:rotation_invariance})
\begin{equation}
\frac{1}{2K}
		\sum_{k=1}^{2K} \log\left( 
		\sum_{u_{k-1} \le n_k \le u_{k+1}^+} \exp\left(2(-1)^k (S_{n_k} -S_{u_k}) \right)
		\right) 
\end{equation}
converges as $K\to\infty$ towards 
\begin{equation}\label{eq:in_proof_th_Elog}
\E\left[\log\left(\sum_{u_{-1}\le n\le u_{+1}^+} e^{-2S_n^{(\gG)}}\right)\right]
\end{equation}
and
\begin{equation}
\frac{\#I_K}{2K}\underset{K\to\infty}{\longrightarrow}
\frac{\p[H_{\gG-c_\gG}^\downarrow \le \gG+c_\gG] +\p[H_{\gG-c_\gG}^\uparrow \le \gG+c_\gG]}{2}\, .
\end{equation}
It follows that 
\begin{equation}
	\begin{aligned}
		& \calF(J) \le 
		\calM(J) + 2 \log\left(1+e^{-2c_\gG}\right) 
		+ \frac{\E[L_{\gG-c_\gG}^\downarrow+L_{\gG-c_\gG}^\uparrow]}{2} \times  \\
		& \left( \E\left[\log\left(\sum_{u_{-1}\le n\le u_{+1}^+} e^{-2S_n^{(\gG)}}\right)\right] 
		+ \frac{\p[H_{\gG-c_\gG}^\downarrow \le \gG+c_\gG] +\p[H_{\gG-c_\gG}^\uparrow \le \gG+c_\gG]}{2} \right) \, .
	\end{aligned}
\end{equation}

Now, as $\gG\to\infty$ we have, due to Corollary \ref{th:cor_Donsker}, that 
\begin{equation}
\frac{\E[L_{\gG-c_\gG}^\downarrow+L_{\gG-c_\gG}^\uparrow]}{2} \sim \frac{(\gG-c_\gG)^2}{\vartheta^2}  \sim \frac{\gG^2}{\vartheta^2}\, ,
\end{equation}
and, due to Fact \ref{th:fact_Donsker}, that
\begin{equation}
\frac{\p[H_{\gG-c_\gG}^\downarrow \le \gG+c_\gG] +\p[H_{\gG-c_\gG}^\uparrow \le \gG+c_\gG]}{2} \to 0\, .
\end{equation}
For the term in \eqref{eq:in_proof_th_Elog}, use Lemma \ref{th:lem_Elog_fine}: this concludes the proof of the upper bound in Theorem \ref{th:dev_calF}.

\section{Proof of Theorem \ref{th:dev_calM}}\label{sec:proof_th:dev_calM}

In this section, we prove Theorem \ref{th:dev_calM}. 
To simplify expressions and avoid recurring factors of 2, we adopt the following notation. As in the previous sections, we work with the parameter $\gG=2J$. Additionally, we define
\begin{equation}
\mathbf{z}_n:= 2 h_n\, , \qquad n=1, 2, \dots 
\end{equation}
and 
\begin{equation}
T_0:=0 \quad \text{ and } \quad  T_n:= \sum_{i=1}^n \mathbf{z}_i = 2S_n , \qquad n=1, 2, \dots 
\end{equation}
We also denote by $\zeta$ the law of $\mathbf{z}_1=2h_1$, i.e., the image of $\mu$ under the dilation map $x\mapsto 2x$. Then, $\mathbf{z}=(\mathbf{z}_n)_{n\in\N}$ is an i.i.d. sequence with marginal law $\zeta$, and $T=(T_n)_{n\ge 0}$ is the corresponding random walk started at 0.

\subsection{Integral formula for $\calM(J)$}\label{sec:integral_form_calM}

Let $M_0^+ = 0$, $M_0^- = -2J$, and for $n\ge 1$, $M_n^+ = M_{n, J, h}^{++}$ and $M_n^- = M_{n, J, h}^{+-}$. By distinguishing according to the value of the spin at position $n$, we get the recurrence relation
\begin{equation}\label{eq:recursion_M+_M-}
\begin{cases}
M^+_{n+1} & =  \max( M_n^+ + h_{n+1},  M_n^- - 2J - h_{n+1} )\, ,\\
M^-_{n+1} & =  \max( M_n^+ - 2J + h_{n+1}, M_n^- - h_{n+1} )\, ,
\end{cases}
\end{equation}
for every $n\ge 0$. Now, let us consider $X_n:= M_n^+ - M_n^-$. We have $X_0=2J$ and we derive from \eqref{eq:recursion_M+_M-} the recurrence relation
\begin{equation}\label{eq:recursion_X}
X_{n+1} = h_\gG ( X_n + \bfz_{n+1})\, ,
\end{equation}
where we recall that $\gG=2J$ and $\bfz_n=2h_n$ and we set
\begin{equation}
h_\gG(x)=
\begin{cases}
-\gG & \text{ if } x\le -\gG \, ,   \\
x & \text{ if } -\gG \le x \le \gG  \, ,  \\
\gG & \text{ if }x\ge \gG\, .
\end{cases}
\end{equation}

Note that the process $X=(X_n)_{n\ge 0}$ is a Markov chain; we denote its transition kernel by $\pi_\gG$. Since $\widehat{h}_\gG$ sends $\R$ to $[-\gG, \gG]$, we may consider that the state space of the chain is reduced to $[-\gG, \gG]$ instead of $\R$.

This chain is positive recurrent ($\gG$ and $-\gG$ are accessible states) and a simple coupling argument shows that it has a unique invariant probability measure. Indeed, for $x, y\in[-\gG, \gG]$, let us couple two copies of the chain $X$, started respectfully at $x$ and at $y$ by driving them by the same i.i.d. sequence $(\logZ_n)_{n\ge0}$; let us assume without loss of generality that $x > y$, then the chains will stay ordered in this way, and they will meet at latest when the chain started from $x$ hits $-\gG$, and will evolve together afterwards. Since, clearly, this will happen almost surely 
(recall that $\zeta$ is centered and non-trivial), the uniqueness of the invariant measure follows. 
We denote by $\nu_\gG$ the unique invariant probability measure of the chain $X$.  

Moreover, rewriting the first equation in \eqref{eq:recursion_M+_M-} as 
\begin{equation} 
  M^+_{n+1} = M_n^+ + h_{n+1} + \max( 0 ,  - X_n - \gG - \bfz_{n+1} ) \, ,
\end{equation} 
and since $M_0^+=0$, we have for all positive integers $N$
\begin{equation}
M_{N}^{+} =
\sum_{n=0}^{N-1} h_{n+1} + \max(0, - X_n -\gG- \bfz_{n+1})\, .
\end{equation}
Diving by $N$ and using the LLN and the ergodic theorem yields the almost sure convergence
\begin{equation}\label{eq:integr_form_calM}
\frac{M_N^{++}}{N} \underset{N\to\infty}{\longrightarrow}
\iint_{\R^2} \max(0, -x-z-\gG) \nu_\gG(d x) \zeta(\dd z)\, .
\end{equation}
We have thus established the second convergence in Lemma \ref{th:def_free_ener} once again, this time with an explicit formula for the limiting maximal energy density, given by the right-hand side of \eqref{eq:integr_form_calM}. Let us denote for every probability measure $\nu$ on $\R$
\begin{equation}\label{eq:definition_calM_gG}
  \calM_\gG[\nu] = \iint_{\R^2} \max(0, -x-z-\gG) \nu(d x) \zeta(\dd z)\, .
\end{equation}
The limiting maximal energy density $\calM(J)$ is equal to $\calM_\gG[\nu_\gG]$.



\subsection{Looking from the edge: the reduced chain $Y$}
\label{sec:dev_max_ener:Y-MC}
If we sit on $-\gG $, that is if we make it our new origin, in the limit as $\gG \to \infty$ the Markov chain becomes
\begin{equation}
\label{eq:dev_max_ener:iterY}
Y_{n+1}\, =\,  \max(Y_n + \logZ_{n+1}, 0 ) \, .
\end{equation}
Since one iteration of the chain sends $Y$ inside $[0, \infty)$, we consider its state space to be $[0, \infty)$. This chain is sometimes refered to in the literature as the Lindley process.

\medskip

\begin{theorem}\label{thm:dev_max_ener:stationary_asymptotics}
Assume \eqref{eq:assump_mu_initial}. 
Then, the Markov chain $Y$ has a unique nonzero invariant 
measure $\nu$. 
This measure $\nu$ satisfies that ${\nu([0,x])< \infty}$ for every $x\ge 0$, and $\nu([0, \infty))= \infty$. 
Moreover, $Y$ is recurrent in the sense that if $O\subset [0, \infty)$ is an open set with $\nu (O)>0$, then, given any $Y_0$,
$\bbP(Y_n \in O \text{ i.o.})=1$.

\smallskip




 
	Furthermore, additionally assuming hypothesis {\bf (H-1)} of Theorem \ref{th:dev_calM} and the following condition: 
	 \begin{description}
	\item[(T)] there exists $\xi > 3$ such that 
	$\int (z^+)^\xi \zeta(\dd z) < \infty$; 
	\end{description}
	the measure $\nu$ satisfies that
	\begin{equation} 
		\label{eq:dev_max_ener:stationary_asymptotics}
		\nu([0, x])\overset{x\to \infty}{=}  c_\nu x+ d_\nu + O(1/x^{\xi-3})\, .
	\end{equation}     
	for appropriately chosen constants $c_\nu>0$ and $d_\nu\in \bbR$. 
	  
	Finally, if {\bf (T)} is replaced by 
	\begin{description}
		\item[(T$^\prime$)] there exists $c>0$ such that $\int \exp(cz) \zeta(\dd z) <  \infty$,
	\end{description}	  
	  \noindent then \eqref{eq:dev_max_ener:stationary_asymptotics} holds with error term $O\big(\exp\big(-\gd x \big)\big)$, for some positive constant $\gd$.
	  \end{theorem}
	 \medskip

For nonnormalizable measures uniqueness is of course meant \emph{up to  a multiplicative constant} and we note that $\nu$ is characterized by 
\begin{equation}
\label{eq:dev_max_ener:fornu}
\int_\R  g(y) \nu (\dd y)\, =\, \iint_\R  g((y+z)_+) \gz(\dd z)  \, \nu (\dd y)\ \ \text{ for every measurable } g \ge 0\,  .
\end{equation}

We stress that, since the invariant measures are infinite, what is relevant is the ratio $d_\nu/c_\nu$, for which the proof will provide the following expression:
\begin{equation}\label{eq:value_ratio_d/c}
\frac{d_\nu}{c_\nu}= \frac{\bbE[H_1^2]}{2\bbE[H_1]}\, ,
\end{equation}
where $H_1=T_{\alpha_1}$, with $\alpha_1=\inf\{n\ge 1: T_n>0\}$.
When applying \Cref{thm:dev_max_ener:stationary_asymptotics} we will set $c_\nu=1$, see \eqref{eq:dev_max_ener:asymptlar}.

\begin{proof}
We observe that 0 is an accessible and recurrent state for the Markov chain $Y$. Indeed, recall that $(\logZ_n)_{n=1, 2, \ldots}$ is an IID sequence with law $\zeta$ which, under \eqref{eq:assump_mu_initial}, is centered and has a positive and finite variance and that $(T_n)$ is the corresponding random walk starting at $0$: $T_0=0$ and $T_n:= \sum_{j=1}^n \bfz_j$ for $n=1,2, \ldots$. Of course this random walk is null recurrent. Let us now consider $Y$ started from some point $y_0\ge 0$. We introduce
\begin{equation}
\tau_0 = \inf \{n \ge 1 : y_0+T_n \le 0\}\, .
\end{equation}
We note that:
\begin{equation}\label{eq:linkx+T-Y}
	\forall \,  0\le n < \tau_0, Y_n=y_0+T_n, \qquad \text{and} \qquad Y_{\tau_0}=0
\end{equation}
so that $\tau_0$ is the return time to 0 of $Y$.
By recurrence of $T$, $\tau_0$ is almost surely finite, hence $Y$ almost surely hits 0. It follows that $Y$ admits a unique invariant measure $\nu$, up to a multiplicative constant, and that $\nu(\{0\})<\infty$.
The above arguments show that the expected time between two visits of $Y$ at 0 is infinite, hence $Y$ is null recurrent and any nonzero invariant measure $\nu$ is infinite. 
Furthermore, the loop representation of $\nu$ yields for every $x\ge 0$:
\begin{equation}
\nu([0, x]) = \nu(\{0\}) \bbE_0\left[\sum_{n=0}^{\tau_0-1} \ind_{Y_n \le x} \right]
\end{equation}
We argue that $\nu([0, x])< \infty$ by pointing out that, for fixed $x$, there exists a time $t$ such that, starting from any point in $(0, x]$, $Y$ has a positive probability of hitting 0 before time $t$, uniformly on the starting point. Consequently we can bound the time spent in $[0, x]$ before the first return time to 0 by $t$ times a geometric variable, hence $\nu([0, x])$ is finite.

Let us now assume {\bf (H-1)} and {\bf (T)} (respectively, {\bf (T$^\prime$)}) to hold and further exploit the loop representation of $\nu$. 
We introduce the strictly ascending ladder times of the random walk $T$:
\begin{equation}
\alpha_0:=0 \, , \quad 
\text{and, for } k\ge 1, \quad 
\alpha_k:=\min\{n\ge \alpha_{k-1}: T_n > T_{\alpha_{k-1}}\} \, , 
\end{equation} 
as well as the associated sequence of strictly ascending ladder heights 
$H_k:=T_{\alpha_k}$.
We note that, a.s., for every $k$, $\alpha_k< \infty$ and $H_k< \infty$.

We use the observation \eqref{eq:linkx+T-Y} with $x=0$ to derive that for $x\ge 0$,
\begin{equation}
\begin{aligned}
\nu([0, x]) & = \nu(\{0\}) \bbE_0\left[\sum_{n=0}^{\tau_0-1} \ind_{Y_n \le x} \right]\\
& = \nu(\{0\})   \sum_{n=0}^\infty  \bbP[T_n \le x, T_1 >0, T_2>0, \dots, T_n>0]\\
& =  \nu(\{0\})  \sum_{n=0}^\infty  \bbP[T_n \le x, T_n>T_{n-1}, \dots, T_n>T_1, T_n>0]\\
& =  \nu(\{0\})  \sum_{k=0}^\infty  \bbP[H_k\le x ]  \, ,
\end{aligned}
\end{equation}
where we used the time reversal of $T$ between the second and the third line.

Now, we can use the renewal lemma to control the sum $\sum_{k=0}^\infty  \bbP[H_k\le x ]$. We refer to Section 4.3 of \cite{CGGH25} because the renewal lemma is applied there exactly in the same situation as for us, except that in \cite{CGGH25}, the $H_k$'s are defined by \emph{weak} (instead of strict) ascending ladder heights. We have, as $x\to\infty$, for some $\delta>0$, 
\begin{equation}
\nu([0, x]) =\nu(\{0\})  \left( \frac{x}{\E[H_1]} +   \frac{\bbE[H_1^2]}{2\bbE[H_1]^2} \right)
+
\begin{cases}
O(1/x^{\xi-3}), \qquad &\mbox{under {\bf (T)}} , \\
O(\exp(-\delta \,x) ) , \qquad &\mbox{under {\bf (T$^\prime$)}} \, .
\end{cases}   
\end{equation}
The proof is therefore complete.
\end{proof}

\subsection{Contraction estimate}\label{sec:dev_max_ener:contraction}

We denote by $d(\mu, \nu)$ the Wasserstein-1 distance between two probability measures $ \mu$ and $\nu$ on $\R$. 
We observe that for every measures $\mu, \nu$ on $\R$, we have that $|\calM_\gG[\mu]-\calM_\gG[\nu]| \le d(\mu, \nu)$, with $\calM_\gG$ defined in \eqref{eq:definition_calM_gG}.

Recall the notation $\pi_\gG$ for the transition kernel of the Markov chain $X$, which is defined by \eqref{eq:recursion_X}.

\begin{lemma}\label{lem:iter}
There exists positive constants $c$ and $C$, and $\gG_0>0$ such that for every $\gG>\gG_0$ and every $N\ge 0$, for every probability measures $\mu$, $\nu$ on $\R$:
\begin{equation}
d(\mu \pi_\gG^n, \nu \pi_\gG^n)\le C \exp\left(-c \frac{n}{\gG^2}\right) \times d(\mu, \nu).
\end{equation}
\end{lemma}
\begin{proof}
Since $h_\gG$ is 1-Lipshitz and sends $\R$ to $[-\gG, \gG]$, we can assume without loss of generality that $\mu$ and $\nu$ are measures on $[-\gG, \gG]$.
Let us pick a coupling $(X_0, X_0')$ of $\mu$ and $\nu$. As in Section \ref{sec:integral_form_calM}, let us couple the Markov chain started at $X_0$ and at $X_0'$ using the same i.i.d. sequence  $(\bfz_n)_n$ (independent of $(X_0, X_0')$):
\begin{equation}
X_{n+1}= h_\gG(X_n+\bfz_{n+1})\, , \qquad X_{n+1}'= h_\gG(X_n'+\bfz_{n+1})\, .
\end{equation}
First, we note that, since $h_\gG$ is 1-Lispschitz, we have that for every $n\ge 0$, $|X_n-X_n'|\le |X_0-X_0'|$.
Secondly, we claim that, using the notation $t_1(\gG)$ introduced in Section \ref{sec:definitions} for the first instant where a drop of $2\gG$ or larger has appeared in the walk $T$, we have $X_n=X_n'$ for every $n\ge t_1(\gG)$. Indeed, at time $t_1(\gG)$ both processes are at $\gG$ and they remain together afterwards.
We thus derive the following bound:
\begin{equation}
d(\mu \pi_\gG^n, \nu \pi_\gG^n)\le \bbP[ t_1(\gG)>n] \times \bbE[|X_0-X_0'|].
\end{equation}
Optimizing over the couplings of $\mu$ and $\nu$ yields
\begin{equation}
d(\mu \pi_\gG^n, \nu \pi_\gG^n)\le \bbP[ t_1(\gG)>n] \times d(\mu, \nu).
\end{equation}
Finally, we use the bound \eqref{eq:uniform_integr_t_1} established in the proof of Corollary \ref{th:cor_Donsker}.
\end{proof}

 Recall that $\nu_\gG$ is the invariant measure of the Markov chain $X$. From Lemma \ref{lem:iter} we derive the following corollary.
\begin{corollary}\label{thm:dev_max_ener:contraction_cor}
There exists $\gG_1>0$ and a constant $C'$ such that for every $\gG>\gG_1$ and every measure $\nu$ on $[-\gG, \gG]$, we have:
\begin{equation}
d(\nu, \nu_\gG) \le C' \gG^2 d(\nu, \nu \pi_\gG)\, .
\end{equation}
\end{corollary}

\begin{proof}
Using the fact that $ \nu \pi_\gG^n \to \nu_\gG$ and Lemma \ref{lem:iter}, we write:
\begin{equation}
\begin{aligned}
d(\nu, \nu_\gG) 
& \le \sum_{n\ge 0} d(\nu \pi_\gG^n, \nu \pi_\gG^{n+1})\\
& \le \sum_{n\ge 0} C \exp\left(-c \frac{n}{\gG^2}\right) \times d(\nu , \nu \pi_\gG)\\
& = C \times \frac{1}{1-\exp\left(- \frac{c}{\gG^2}\right)} \times d(\nu , \nu \pi_\gG)
\end{aligned}
\end{equation}
Finally, we use the fact that for large $\gG$ we have $\frac{1}{1-\exp\left(-\frac{c}{\gG^2}\right)} \le 2 \frac{\gG^2}{c}$.
\end{proof}

\subsection{Approximating the invariant probability $\nu_\gG$}

We are now going to patch together the  (infinite!) invariant measure of the $Y$ process, translated to the left of $\gG$,
and  the  (infinite too, of course) invariant measure of the $Y$ process driven by $-\logZ$ instead of  $\logZ$, reflected with respect to the origin
and translated to the right of $\gG$ 
to form a probability measure that we call $\gamma_\gG$.
We expect (and will show) that $\gamma_\gG$ is close to the invariant probability $\nu_\gG$ of the Markov chain $X$. 



Let us introduce some notations for the cumulative mass functions. 
If $\nu$ is a measure on $\R$, such that for every $x\in\R$, $\nu((-\infty, x])<\infty$ we denote its cumulative mass function by $F_\nu$:
\begin{equation}
F_\nu(x) := \nu((-\infty, x]) \, , \qquad x \in \R\, .
\end{equation}
If $X$ is a real random variable, we set 
\begin{equation}
F_X(x) : = \bbP[X \le x]  \, , \qquad x \in \R\, .
\end{equation}

Let $F_\lar(x) := c_\nu^{-1} \nu ([0, x])$ for $x \in \bbR$, where $\nu$ and $c_\nu$ are as in \cref{thm:dev_max_ener:stationary_asymptotics}. So $F_\lar(\cdot)$ is the cumulative function of the (suitably normalized) stationary measure of $Y$  
and Theorem~\ref{thm:dev_max_ener:stationary_asymptotics} yields
\begin{equation}
	\label{eq:dev_max_ener:asymptlar}
	F_\lar(x)\stackrel{ x \to \infty}= x+ c_\lar + \begin{cases} O\left(x^{-(\xi-3)}\right) & \text{ under {\bf (T)}} \, , 
	\\
	O\left(\exp(-\gd x)\right) & \text{ under {\bf (T$^\prime$)}}\, , 
	\end{cases}  
\end{equation}
where $ c_\lar := \frac{\bbE[H_1^2]}{2\bbE[H_1]} $. 
Let $F_\rar$ be the same with $\logZ$ replaced by $-\logZ$, which describes the related object for the right edge instead of the left edge, which then has comparable asymptotic behavior: denote by $c_\rar$ the corresponding constant in \eqref{eq:dev_max_ener:asymptlar}. 
Using this, we define an approximate stationary distribution $\gamma_\gG$ on $[-\R, \R]$ by 
\begin{equation}
\label{eq:dev_max_ener:DHprobability_F}
F_{\gamma_\gG} (x)\, :=\, 
\begin{cases} 
1-\left(F_{\rar}(\gG -x) / C_\gG \right)  & \text{ if } x \ge 0\, ,
\\
F_\lar(x+\gG )/C_\gG  & \text{ if } x \le 0\, .
\end{cases}
\ \ \ \text{ with } C_\Gamma :=
	F_{\lar}(\gG) + F_{\rar}(\gG)\,.
\end{equation}
Note that $C_\Gamma$ has been selected so that $F_{\gamma_\gG}(0)$ is coherently chosen.
We register also that for the risk function probability  $G_{\gamma_\gG} (\cdot)=1-F_{\gamma_\gG} (\cdot)$ we have
\begin{equation}
\label{eq:dev_max_ener:DHprobability_G}
G_{\gamma_\gG} (x)\, :=\, \begin{cases} 
F_{\rar}(\gG -x) / C_\gG   & \text{ if } x \ge 0\, ,
\\
1- \left(F_\lar(x+\gG )/C_\gG \right) & \text{ if } x \le 0\, ,
\end{cases}
\end{equation}

Set
\begin{equation}
\label{eq:dev_max_ener:kappa1}
\widehat{\kappa}_1 : = \frac{1}{2} \iint_{\R^2} \max(0, -x-z) \nu_\lar(x) \gz(\dd z)
\end{equation}
and
\begin{equation}
\label{eq:dev_max_ener:kappa2}
\widehat{\kappa}_2\, :=\, \frac{c_\lar +c_\rar}2\, .
\end{equation}

\medskip

\begin{proposition}
\label{thm:dev_max_ener:onestep}
	Assume that both $\logZ$ and $-\logZ$ satisfy the assumptions of Theorem \ref{thm:dev_max_ener:stationary_asymptotics} (in particular, there exists $\xi>3$ such that $\logZ\in \bbL^{\xi}$, so {\bf (T-2)} becomes superfluous).
	Then for  a suitably chosen $\gd>0$
\begin{equation}
\label{eq:dev_max_ener:Ck}
	C_\Gamma 
	\stackrel{\Gamma \to \infty}=
	2\gG  + 2 \widehat{\kappa}_2 + \begin{cases} O\left(1/\gG^{\xi -3}\right) & \text{ under {\bf (T)} for both $\pm \logZ$}, 
	\\
	O\left(\exp(-\gd \gG)\right) & \text{ under {\bf (T$^\prime$)} for both $\pm \logZ$ }.
	\end{cases} 
\end{equation}  
Moreover
\begin{equation}
		d(\gamma_\gG, \gamma_\gG \pi_\gG)		\,=\,
		\begin{cases} O\left(1/\gG^{\xi-2}\right) & \text{ under {\bf (T)} for both $\pm \logZ$}, 
	\\
	O\left(\exp(-\gd \gG)\right) & \text{ under {\bf (T$^\prime$)} for both $\pm \logZ$}.
	\end{cases} 
		\label{eq:dev_max_ener:onestep}
	\end{equation}
	and
	\begin{equation}
		\label{eq:dev_max_ener:Lyap}
		\calM_\gG [\gamma_\gG ]\,=\, \frac {2\widehat{\kappa}_1}{C_\gG }+ 
		\begin{cases} 
		O(1/\gG^{\xi-2}) &  \text{ under {\bf (T)} for both $\pm \logZ$} , \\
		O(\exp(-c\gG )) &  \text{ under {\bf (T$^\prime$)} for both $\pm \logZ$}  .
		\,,
		\end{cases}
	\end{equation}	
\end{proposition}

\medskip
Before proving Proposition \ref{thm:dev_max_ener:onestep}, let us denote by $\ast$ the convolution product: if $f$ is a measurable fonction on $\R$,
\begin{equation}
(f\ast \gz) (x) := \int_\R f(x-z)  \gz(\dd z)\, .
\end{equation}
and let us observe that an iteration of the chain $(X_n)_n$ is encoded in terms of cumultative mass function as:
\begin{equation}
\label{eq:dev_max_ener:iter_F_X}
F_{X_{n+1}}(x)\, =\, 
 \begin{cases}
  0 & \text{if } x < -\gG \, \\
 (F_{X_n} \ast \gz) (x) & \text{if } -\gG \le x \le \gG \, \\
 1 & \text{if } x \ge \gG \, ,
 \end{cases}
\end{equation}
while an iteration of the chain $(Y_n)_n$ is encoded by:  
\begin{equation}
\label{eq:dev_max_ener:Fn-onesided}
F_{Y_{n+1}}(x)\, =\, 
 \begin{cases}
 (F_{Y_n} \ast \gz) (x) & \text{if } x\ge 0\\
 0 & \text{otherwise}.
 \end{cases}
\end{equation}

\begin{proof} [Proof of Proposition \ref{thm:dev_max_ener:onestep}]
\eqref{eq:dev_max_ener:Ck}  follows readily from \eqref{eq:dev_max_ener:asymptlar} and its analogue for $F_\rar$.

Now for the proof of \eqref{eq:dev_max_ener:onestep}.
We refer to assumption  ``{\bf (T)} for $\pm \logZ$'' as ``$\logZ\in \bbL^\xi$'' and to ``{\bf (T$^\prime$)} for $\pm \logZ$''
as ``exponential tails assumption'' and, when the assumption is omitted, $\logZ\in \bbL^\xi$ is assumed. Let us recall a useful identity concerning the Wasserstein-1 distance between two measures $\mu$ and $\nu$ on $\R$:
\begin{equation}
d(\mu, \nu)= \int_\R \left|F_\mu(x) - F_\nu(x) \right|  \dd x = \| F_\mu - F_\nu \|_1 \, ,
\end{equation} 
with $\| \cdot \|_1$ denoting the usual $L^1$-distance. 
This yields:
\begin{equation}
\label{eq:dev_max_ener:2addenda-2}
\begin{aligned}
d(\gamma_\gG, \gamma_\gG \pi_\gG ) 
& =
\left\Vert F_{\gamma_\gG}  - F_{\gamma_\gG \pi_\gG}  \right\Vert_1\, \\
&=\, 
\left\Vert ( F_{\gamma_\gG}  - F_{\gamma_\gG \pi_\gG}  )\ind_{(-\infty,0)}\right\Vert_1 + \left\Vert ( F_{\gamma_\gG}  - F_{\gamma_\gG \pi_\gG}  )\ind_{(0, \infty)}\right\Vert_1
\\
&=\, 
\left\Vert ( F_{\gamma_\gG}  - F_{\gamma_\gG \pi_\gG}  )\ind_{(-\infty,0)}\right \Vert_1 + \left\Vert ( G_{\gamma_\gG}  - G_{\gamma_\gG \pi_\gG}  )\ind_{(0, \infty)}\right \Vert_1
\, .
\end{aligned}
\end{equation}
where $G_\mu$ denotes the risk function of the measure $\mu$: $G_\mu=1-F_\mu$.
If $\zeta$ is symmetric the two terms coincide; in general, they do not coincide but they can be treated 
in the same way because one can be obtained from the other by replacing $\logZ$ with $-\logZ$.  
So we focus on the first one, which is equal to
\begin{equation}
	A_{\gG ,1}\, :=\,\int_{-\gG}^0 \left \vert   F_\gG(x) - \int_\R  F_\gG( x-z)  \gz( \dd z) \right\vert  \dd x\, , 
\end{equation}
with the abbreviation $F_\gG (\cdot) = F_{\gamma_\gG} (\cdot)$, given in 
\eqref{eq:dev_max_ener:DHprobability_F}. Now, we use the fact that for $-\gG \le x\le 0$
\begin{equation}
F_\gG(x)  = \frac{ F_\lar(x+\gG ) }{C_\gG} = \frac{1}{C_\gG} \int_{\R} F_\lar(x+\gG -z )  \gz(\dd z)\, ,
\end{equation}
(by definition of $F_\lar$), so
\begin{equation}
A_{\gG ,1}\, :=\, 
\int_{-\gG}^0  
\left \vert 
\int_{-\infty}^x
\frac{F_\rar(\gG -x+z) + F_\lar(\gG +x-z) }{ C_\gG } -1
\quad \gz(\dd z) 
\right\vert 
\dd x\, .
\end{equation}

Now, by \eqref{eq:dev_max_ener:asymptlar} and \eqref{eq:dev_max_ener:Ck},  we have
\begin{equation}
 \left \vert \frac{F_\rar(\gG -y) + F_\lar(\gG +y) }{ C_\gG } -1 \right \vert \,\le\, C \times  \begin{cases}
	\gG ^{-\xi+2} 
	& \text{ if } y\in [0, \gG /2] \, ,
\\
y/\gG  & \text{ if } y >\gG  /2\, ,
\end{cases}
\end{equation}
for suitably chosen $C>0$ and for $\gG $ sufficiently large: in the exponential case the first line of the estimate may be replaced by $\exp(-\gd \gG/2)/\gG$. So 
 $A_{\gG ,1} $ is bounded above by a constant times 
\begin{equation}
	\label{eq:dev_max_ener:interm2.k}
	\begin{split}
		&
		\gG ^{-\xi+2}
		\int_{-\gG}^0 \int_{x-\gG/2}^x 1 \zeta(\dd z) \dd x
		+
		\gG^{-1} 
		\int_{-\gG}^0 \int_{-\infty}^{x-\gG/2} (x-z) \zeta(\dd z) \dd x
		\\ &
		\le
		\gG ^{-\xi+2}
		\int_{-\infty}^0 F_\zeta(x) \dd x
		+\int_{-\infty}^{-\gG/2} (-z) \zeta(\dd z) \, .
	\end{split}
\end{equation}
Now, the quantity $\int_{-\infty}^0 F_\zeta(x) \dd x$ is finite since $\gz$ has a second moment, and the Markov inequality implies, in the polynomial case, that 
$\int_{-\infty}^{-\gG/2} (-z) \zeta(\dd z) =O(\gG^{-\xi+1})$ 
and, in the exponential case, that
 \begin{equation}
\int_{-\infty}^{-\gG/2} (-z) \zeta(\dd z) \le \exp(-c\, \frac{\gG}{2})
\end{equation}
for some $c>0$, so
\begin{equation}
	A_{\gG ,1}
	=\begin{cases}
	O\left( \gG ^{-\xi+2} \right) & \text{ if } \logZ \in \bbL^\xi, \\
	O(\gG \exp(-(\gd \wedge c) \gG/2)) & \text{ under exponential tails assumption}.
	\end{cases}
\end{equation}

Finally, let us prove \eqref{eq:dev_max_ener:Lyap}.

First, we rewrite $\calM_\gG[\nu]$ for $\nu$ a measure on $\R$. For every $x\in \R$
\begin{equation}
\begin{aligned}
\int_\R \max(0, -x-\gG-z) \gz (\dd z) 
& = \int_{-\infty}^{-x-\gG} (-x-\gG-z)  \gz (\dd z) \\
& = [ (-x-\gG-z) F_\gz(z)]_{-\infty}^{-x-\gG} + \int_{-\infty}^{-x-\gG} F_\gz(z) \dd z\, ,
\end{aligned}
\end{equation}
hence
\begin{equation}\label{eq:dev_max_ener:rewrite_L_IPP}
\calM_\gG[\nu]= \int_\R F_\nu(-z-\gG) F_\gz(z) \dd z =  \int_\R F_\nu(z-\gG) F_\gz(-z) \dd z\, .
\end{equation}

We use this identity for $\nu=\nu_\gG$ and we rewrite $\calM_\gG[\gamma_\gG]$ as:
\begin{equation}
\begin{aligned}
\calM_\gG[\gamma_\gG]
 & = \frac{1}{C_\gG} \int_{-\infty}^\gG F_\lar (z) F_\xi(-z) \dd z + \int_{\gG}^\infty F_{\gamma_\gG} (z) F_\xi(-z) \dd z \\
 & = \frac{1}{C_\gG} \int_{\R} F_\lar (z) F_\xi(-z) \dd z  - \frac{1}{C_\gG} \int_{\gG}^\infty F_\lar (z) F_\xi(-z) \dd z + \int_{\gG}^\infty F_{\gamma_\gG} (z) F_\xi(-z) \dd z \, . 
\end{aligned}
\end{equation}
Using the formula \eqref{eq:dev_max_ener:rewrite_L_IPP} for $\nu=\nu_\lar$, we readily see that the first term in this expression is equal to $\frac {2\widehat{\kappa}_1}{C_\gG }$. Hence there only remains to bound the two other terms.
Using the bound $F_{\gamma_\gG} (z) \le 1$, the third term is upper bounded by $\int_{\gG}^\infty   F_\gz(-z) \dd z $. If $\gz$ has a $\xi$-th moment, then the Markov property yields that $F_\gz(-z) = O(z^{-\xi})$ as $x\to\infty$, thus $\int_{\gG}^\infty   F_\gz(-z) \dd z = O(\gG^{-\xi+1})$. In the exponential case, we get instead the bound $\int_{\gG}^\infty   F_\gz(-z) \dd z = O(e^{-c\gG})$.

Now, let us bound $\int_\gG^\infty F_\lar (z) F_\xi(-z) \dd z$. We use the existence of a constant $C$ such that for large $z$, $ F_\lar (z)\le C z$, which yields that for large $\gG$, $\int_\gG^\infty F_\lar (z) F_\xi(-z) \dd z$ is bounded by a constant times
\begin{equation}
\int_\gG^\infty z F_\xi(-z) \dd z
\end{equation}
Proceeding as before, we bound this quantity by $O(\gG^{-\xi +2})$ in the polynomial case and by $O(\gG e^{-c \gG})$ in the exponential case.
\end{proof}

\subsection{Proof of Theorem \ref{th:dev_calM}}


With Proposition \ref{thm:dev_max_ener:onestep} at hand, it is immediate to derive the proof of Theorem \ref{th:dev_calM}.

\begin{proof}[Proof of Theorem \ref{th:dev_calM}]
Using the observation in the beginning of Section  \ref{sec:dev_max_ener:contraction} and Corollary \ref{thm:dev_max_ener:contraction_cor},
we have
\begin{equation}
|\calM_\gG[\nu_\gG] - \calM_\gG[\gamma_\gG]| \le d(\nu_\gG, \gamma_\gG) \le C'' \gG^2 d(\gamma_\gG, \gamma_\gG \pi_\gG) \,,
\end{equation}
We conclude the proof using the estimates given in Proposition \ref{thm:dev_max_ener:onestep}. We get Theorem \ref{th:dev_calM} with constants $\widehat{\kappa_1}$ in the numerator and $\widehat{\kappa_2}$ in the denominator. We use Theorem \ref{th:first_order_calM} to identify $\widehat{\kappa}_1$ with $\vartheta^2$. Concerning $\widehat{\kappa_2}$ (given by \eqref{eq:dev_max_ener:kappa2}, see also \eqref{eq:value_ratio_d/c}), rewritting it in terms of the random walk $S$ instead of $T$ yields \eqref{eq:expr_widehat_kappa}.
\end{proof}

\begin{remark}
The constant $\widehat{\kappa}$ is given in \eqref{eq:expr_widehat_kappa} in terms of strictly ascending record heights. Introducing 
\begin{equation}
	\rho_\lar := \inf\{n\ge 1: S_n\ge 0\}  \qquad \text{and} \qquad \rho_\rar := \inf\{n\ge 1: S_n\le 0\}\, ,
\end{equation}
we have the alternate expression
\begin{equation}\label{eq:widehatkappa_alternate_expre}
	\widehat{\kappa}= \frac{1}{2} 
	\left(
	\frac{\E\left[ (S_{\rho_{\lar}})^2\right]}{\E\left[ S_{\rho_{\lar}} \right]}
	+ \frac{\E\left[ (S_{\rho_{\rar}})^2\right]}{\E\left[ -S_{\rho_{\rar}} \, .\right]}
	\right)
\end{equation}
Indeed, we observe that the strictly ascending excursion $(S_n)_{0\le n\le \alpha_\lar}$ can be decomposed into a succession of weakly ascending excursions: a (geometric) number of excursions with overshoot zero and a final excursion conditioned on having a positive overshoot. Therefore:
\begin{equation}
	\E\left[ S_{\alpha_\lar} \right] 
	= \E\left[ S_{\rho_\lar} | S_{\rho_\lar}>0 \right] 
	= \frac{\E\left[ S_{\rho_\lar}\right]}{\p\left[S_{\rho_\lar}>0\right]} \,  , 
\end{equation}
and 
\begin{equation}
	\E\left[ \left(S_{\alpha_\lar} \right)^2 \right] 
	= \E\left[ \left(S_{\rho_\lar} \right)^2  | S_{\rho_\lar}>0 \right] 
	= \frac{\E\left[ \left(S_{\rho_\lar} \right)^2 \right]}{\p\left[S_{\rho_\lar}>0\right]}\, . 
\end{equation}
We reason similarly for $S_{\alpha_\rar}$ 
and deduce \eqref{eq:widehatkappa_alternate_expre}.
\end{remark}






\section*{Acknowledgements}
The author thanks Yueyun Hu for a useful discussion on the proof of Lemma \ref{th:lem_Elog_fine}.

\bigskip

This research was funded in part by the Austrian Science Fund (FWF) 10.55776/F1002. For open access purposes, the author has applied a CC BY public copyright license to any author accepted manuscript version arising from this submission.

\bibliographystyle{plainnat}
\bibliography{bibliography}


\end{document}